\documentclass[a4paper, 11pt]{amsart}
\usepackage[utf8]{inputenc}
\usepackage{dsfont}
\usepackage{amssymb}
\usepackage{amsfonts}
\usepackage{amsthm}
\usepackage{amsmath}
\usepackage{verbatim}
\usepackage{a4wide}
\usepackage{bm}
\usepackage{color}
\usepackage{appendix}
\usepackage{bbm}
\usepackage{fancybox}
\usepackage{fancyhdr}
\usepackage{graphicx}
\usepackage{tristan}
\usepackage{subfigure}
\usepackage[cm]{fullpage}
\usepackage{pgfplots}
\pgfplotsset{width=7cm,compat=1.8}
\usepackage{pgfplotstable}
\usepgfplotslibrary{fillbetween}
\usepackage{multirow}
\usepackage{booktabs}
\usepackage{longtable}

\newtheorem{theorem}{Theorem}[section]
\newtheorem{lemma}[theorem]{Lemma}

\newtheorem{assumption}[theorem]{Assumption}
\newtheorem{corollary}[theorem]{Corollary}
\theoremstyle{definition}
\newtheorem{definition}[theorem]{Definition}
\theoremstyle{definition}

\newtheorem{remark}[theorem]{Remark}
\theoremstyle{definition}

\usepackage[colorlinks,citecolor=cyan,hypertexnames=false]{hyperref}

{\catcode `\@=11 \global\let\AddToReset=\@addtoreset}
\AddToReset{equation}{section}

\newcommand{\veps}{{\varepsilon }}

\newcommand{\vut}{\widehat{\vec u}_h^{\mathrm t}}
\renewcommand{\ut}{\widehat{u}_h^{\mathrm t}}
\newcommand{\vuts}{\widehat{\vec u}^{\mathrm{ts}}}
\newcommand{\uts}{\widehat{u}^{\mathrm{ts}}}

\newcommand{\bg}{{\mathcal C_h}}
\newcommand{\bD}{{\mathcal D_h}}

\newcommand{\del}{\partial}

\renewcommand{\vec}[1]{\geovec{#1}}
\newcommand{\hatu}{\widehat{\vec u}}

\newcommand{\strec}{\widehat{\vec u}^{\mathrm{ts}}}

\renewcommand{\d}{\ensuremath{\,\mathrm{d}}}

\author{Andreas Dedner}
\address{Andreas Dedner \newline 
Mathematics Institute\newline
Zeeman Building\newline
University of Warwick\newline
Coventry CV4 7AL\newline
UK}
\curraddr{}
\email{a.s.dedner@warwick.ac.uk}
\author{Jan Giesselmann} 
\address{Jan Giesselmann\newline
Technical University of Darmstadt\newline
Department of Mathematics\newline
Dolivostr 15\newline
D-64293 Darmstadt\newline
Germany} 
\curraddr{}
\email{giesselmann@mathematik.tu-darmstadt.de}
\thanks{The research of J.G. and K.K. was supported by Deutsche Forschungsgemeinschaft (DFG, German Research Foundation) - SPP 2410 Hyperbolic Balance Laws in Fluid Mechanics: Complexity, Scales, Randomness (CoScaRa), within the Project ``A posteriori error estimators for statistical solutions of barotropic Navier-Stokes equations'' 525877563. The work of J.G. is also supported by the Graduate School CE within Computational Engineering at Technische Universität Darmstadt.}
\author{Kiwoong Kwon}
\address{Kiwoong Kwon\newline
Technical University of Darmstadt\newline
Department of Mathematics\newline
Dolivostr 15\newline
D-64293 Darmstadt\newline
Germany}
\email{kwon@mathematik.tu-darmstadt.de}
\author{Tristan Pryer}
\address{
Tristan Pryer,\newline
Department of Mathematics, University of Bath, Bath BA2 7AY, UK
}
\email{tmp38@bath.ac.uk}
\thanks{
TP gratefully acknowledges the support of the EPSRC grant EP/P000835/1.
}

\title[A posteriori analysis for convection-diffusion systems]
{A posteriori error estimates for Runge-Kutta discontinuous Galerkin approximations of nonlinear convection-diffusion systems}
\date{\today}

\begin{document}

\begin{abstract}
  This work provides reliable a posteriori error estimates for
  Runge-Kutta discontinuous Galerkin approximations of nonlinear
  convection-diffusion systems. {The classes of systems we study are
  quite general, with a focus on convection-dominated problems and systems
  with partial diffusion.} Our a posteriori error bounds are valid for a family
  of discontinuous Galerkin spatial discretizations and various
  temporal discretizations that include explicit and implicit-explicit
  time-stepping schemes, popular tools for practical simulations of
  this class of problem.
  We prove that our estimators provide reliable upper bounds for
  the error of the numerical method and present numerical evidence
  showing {that, in the convection-dominated regime, the residual quantities converge at the same rates as the discretization errors.}
  Since one of our main interests is the convection dominant case, we also track the dependence of the estimator on the viscosity coefficient.
\end{abstract}

\maketitle

\section{Introduction}\label{sec:intro}
The goal of this work is to provide reliable a posteriori error
estimates for numerical approximations of nonlinear
convection-diffusion systems. We discretize these systems using the method of lines
approach, employing discontinuous Galerkin (dG) methods for spatial discretization. A particularly popular sub-class is
the Runge-Kutta discontinuous Galerkin (RKdG) scheme
\cite{CockburnShu:1998}. 

Our primary focus is convection-dominated flows, particularly those
that remain stable in the vanishing viscosity limit. We prove that our
estimators provide reliable upper bounds for the error of the
numerical method and present numerical evidence showing
{that, in the convection-dominated regime, the residual quantities converge at the same rates as the discretization errors.}

Extensive literature exists on a posteriori error estimates for continuous and discontinuous Galerkin
methods for convection-diffusion equations, with a focus on scalar problems and linear
advection. A non-exhaustive list includes \cite{Verfurth:2013, Ver05,
Sangalli:2008, Kunert:2003, DEV13, TV15, CGM14}.
{Sharp upper and lower a posteriori error bounds in energy-type norms have also been derived using elliptic reconstruction for fully discrete approximations of linear parabolic problems \cite{GeorgoulisMakridakis:2023}.}

A crucial component in a posteriori analysis is the underlying stability framework of the partial differential equation (PDE) or PDE system being studied. Traditional analysis of convection-diffusion equations relies on the parabolic component through energy estimates. However, for nonlinear systems in the convection-dominated regime, this approach may fail to yield a posteriori estimators that are robust with respect to the viscosity parameter.

Instead, we consider nonlinear convection-diffusion problems from a
\emph{hyperbolic} viewpoint. This perspective is particularly well suited
to nonlinear systems in the convection-dominated regime, which are our
primary focus. The stability analysis presented here is based on the
relative entropy technique, which originated in the work of Dafermos
\cite{Daf79} and DiPerna \cite{Dip79} on nonlinear hyperbolic systems and
has since been extended to a broad range of conservative and dissipative
models \cite{FN07, GLT17, Tza05}.

The generality of this framework allows us to examine how the estimates
specialize to linear and scalar problems and to compare them with existing
results. In particular, for the linear wave system with partial diffusion
considered in Section~\ref{sec:lw}, the framework yields a posteriori error
bounds without the exponential factor that appears in the corresponding
nonlinear estimate. These bounds can be compared with existing a posteriori
estimates for the wave equation \cite{GeorgoulisLakkisMakridakis:2013,
GeorgoulisLakkisMakridakisVirtanen:2016, 2025Dong}.

{In the vanishing viscosity limit, our a posteriori error bounds reduce to the bounds for approximations of systems of hyperbolic conservation laws derived in \cite{GMP_15, DG16}.}
Our analysis also accommodates quite general time integration methods, including explicit and implicit-explicit (IMEX) methods that are popular in practical simulations for this class of problems. 
For clarity, we focus on Runge-Kutta methods up to third order, though the methodology extends naturally to higher-order schemes.

In this work we study one-dimensional {convection-diffusion} systems of the form
\begin{equation}\label{eq:cd-intro}
  \partial_t \vec u + \partial_x \vec f(\vec u)
  =
  \veps \partial_x (\vec A (\vec u) \partial_x \vec u), \quad (t,x) \in [0,T] \times \W,
\end{equation}
where $\W \subset \reals$ is a bounded {interval} and $T > 0$. {Here, $U\subset\reals^m$ is an open set, $\vec u(t,x)\in U$ is the state vector, $\vec f:U\rightarrow\reals^m$ is the nonlinear convective flux, and $\vec A:U\rightarrow\reals^{m\times m}$ is the diffusion matrix, which may depend nonlinearly on $\vec u$.} The parameter $\veps \geq 0$ is the viscosity coefficient. {When $\veps=0$, the system reduces to a hyperbolic conservation law, whereas for $\veps>0$ we consider both systems in which diffusion acts on every component and systems with partial diffusion. In the latter case, diffusion acts only on selected components, so $\vec A(\vec u)$ may be degenerate.} We focus particularly on the convection-dominated regime where $\veps \ll 1$. {We impose boundary conditions according to the problem under consideration, generally adopting periodic boundary conditions and otherwise using Dirichlet boundary conditions for viscous problems or inflow boundary conditions for inviscid problems.}

The relative entropy framework requires that this system is endowed
with an \emph{entropy/entropy-flux} pair (see \S \ref{sec:gref} for a
precise definition). This induces an additional balance law that is
satisfied by classical solutions to \eqref{eq:cd-intro}. 
It should be noted that scalar equations and {many, but not all,} thermomechanical systems possess such
entropy/entropy-flux pairs.
In particular, for thermomechanical theories, the additional balance law arises from the second law of thermodynamics \cite{Daf10}.

{In addition to the existence of an entropy/entropy-flux pair, the relative entropy analysis requires the entropy to be strictly convex, a condition that may not hold in settings such as multiphase flows and nonlinear elastodynamics in multiple space dimensions.}

{Another useful technique in a
posteriori analysis is to use a reconstruction of the numerical solution.
By a reconstruction, we mean a function constructed from the numerical
solution with the regularity needed to apply the PDE stability estimate.
Such a function may be introduced only for the analysis, as in the elliptic
reconstruction approach \cite{Lakkis:2006}, or may be fully computable from
the numerical solution.}
The underlying idea of obtaining a posteriori error estimates by combining a reconstruction of the numerical solution
with an appropriate PDE stability estimate goes back to \cite{Makridakis:2003} and has since been extensively used.

Our error estimate strategy is based
on interpreting a {\it reconstruction} $\hatu$ of the numerical
solution as the solution of a perturbed equation, i.e.,
\begin{equation}\label{eq:cdr-intro}
  \partial_t \hatu + \partial_x \vec f(\hatu)
  =
  \veps \partial_x (\vec A(\hatu) \partial_x \hatu) + \vec r.
\end{equation}
In fact, the reconstruction $\hatu$ we propose is a fully computable space-time reconstruction, and thus equation \eqref{eq:cdr-intro} can be understood as a definition of the residual $\vec r$. We then apply
the relative entropy framework to bound the difference between $\vec u$ and $\hatu$ in terms of $\vec r$.

We note that there are other methodologies to construct a posteriori indicators, not only for the class of problems we consider but also for more general systems such as compressible fluid flows \cite{HH02,DRV15,GiesselmannPryer:2017}. These alternative approaches are typically based on duality arguments.

{A key challenge in the analysis of nonlinear convection-diffusion systems arises when diffusion acts only on certain variables, since solution regularity then becomes delicate and standard parabolic theory cannot be applied directly. In this setting, as in the purely hyperbolic case, classical solutions may exist only locally in time, and weak solutions need not be unique. The additional conservation law associated with the entropy yields entropy inequalities that serve as admissibility criteria for excluding nonphysical weak solutions. Nevertheless, these inequalities do not guarantee uniqueness in general. For certain initial data, the incompressible Euler equations admit multiple weak solutions satisfying the relevant energy inequalities, and an analogous result holds for the p-system of isentropic gas dynamics in more than one space dimension \cite{DS10}. The relative entropy framework accordingly yields only \emph{weak-strong} stability. An entropy solution can be compared with a Lipschitz solution, but the argument does not control the difference between two discontinuous entropy solutions. {The space-time reconstruction is Lipschitz and can therefore be used in this comparison. At each $t\in[0,T]$, the error measured up to time $t$ is bounded from above by the resulting a posteriori error bound, so this bound cannot be small when that error is large. For nonlinear problems, this reliability does not imply sharpness because the bound contains a stability factor depending on the Lipschitz constant of the reconstruction and may become pessimistic when this constant grows. Therefore, the present analysis does not guarantee convergence of the full a posteriori error bound under mesh refinement if the exact solution is not Lipschitz. Diffusion may improve regularity, but it does not remove this limitation in the convection-dominated and partially diffusive cases.}}

{The paper is organized as follows. In Section \ref{sec:nsr}, we describe the numerical schemes and the space-time reconstruction procedure used in the relative entropy analysis. In Section \ref{sec:gref}, we present a general relative entropy framework for convection-diffusion problems. In Sections \ref{sec:ls}--\ref{sec:yl}, we derive a posteriori error bounds for linear scalar equations, nonlinear scalar equations, and linear systems. In Sections \ref{sec:lw} and \ref{sec:nw}, we consider linear and nonlinear wave systems with partial diffusion, respectively. In Section \ref{sec:num}, we examine the scaling of the residual quantities and the robustness of the estimators with respect to viscosity through numerical experiments.}

\section{Numerical schemes and reconstructions}
\label{sec:nsr}

{We consider fully discrete Runge-Kutta discontinuous Galerkin (RKdG) approximations of \eqref{eq:cd-intro}. We first discretize the equation in space using a dG method and then apply a Runge-Kutta method to the resulting system of ordinary differential equations. We then describe how temporal and space-time reconstructions with the regularity required for the relative entropy analysis are obtained from the numerical solutions.}

\subsection{RKdG schemes}
{
Let $\W=(0,1)\subset\reals$ be partitioned by the grid points
\[
0=x_0<x_1<\cdots<x_{M-1}<x_M=1.
\]
{For simplicity of presentation, we formulate the numerical scheme under periodic boundary conditions. We identify the endpoints $x_0$ and $x_M$ and denote the periodic domain by $\mathbb T$. For other boundary conditions, the periodic coupling between $x_0$ and $x_M$ is replaced by boundary terms determined by the prescribed data.}
}

{
For {$k=0,\ldots,M-1$}, set
\[
h_{k+\frac12}:=x_{k+1}-x_k.
\]
We use the convention {$h_{-\frac12}:=h_{M-\frac12}$} and set
\[
h_k:=\frac12\left(h_{k+\frac12}+h_{k-\frac12}\right),
\qquad {k=0,\ldots,M-1}.
\]
We write
\[
h:=\max_{0\leq k\leq M-1}h_{k+\frac12}.
\]
}

{
For $p\in\mathbb N$ and an interval $I$, let $\mathbb P_p(I,\reals^m)$ denote the space of $\reals^m$-valued polynomials of degree at most $p$ on $I$. We use
\begin{equation}\label{def:dgs}
\mathbb V_h^p
:=
\left\{
\vec w\in\leb{2}(\W;\reals^m):
\left.\vec w\right|_{(x_{i-1},x_i)}
\in\mathbb P_p((x_{i-1},x_i),\reals^m),
\quad i=1,\ldots,M
\right\}
\end{equation}
for both the dG trial and test functions.
}

{
Let $\bg:\mathbb{V}_h^p\rightarrow\mathbb{V}_h^p$ and
$\bD:\mathbb{V}_h^p\rightarrow\mathbb{V}_h^p$ denote the discrete convection
and diffusion operators approximating $\partial_x\vec f(\vec u)$ and
$\partial_x\bigl(\vec A(\vec u)\partial_x\vec u\bigr)$ in
\eqref{eq:cd-intro}, respectively.
}
{
We seek a semi-discrete dG approximation $\vec u_h:[0,T]\to\mathbb V_h^p$ satisfying, for almost every $t\in(0,T]$,
}
\begin{align}
  \label{eq:sdisc}
    \partial_t {\vec u}_h
    +
    \bg ({\vec u}_h)
    &=
    \veps
    \bD (\vec u_h)
    \\
    \vec u_h(0) &= \mathbb{P} \vec u(0), \nonumber
\end{align}
{
where $\mathbb{P}:\leb{2}(\mathbb T;\reals^m)\to\mathbb V_h^p$ denotes the $\leb{2}$-projection characterized by
\[
\int_{\mathbb T}(\mathbb P\vec v-\vec v)\cdot\vec\psi_h\,\d x=0
\qquad\text{for all }\vec\psi_h\in\mathbb V_h^p.
\]
}

{
The discrete convection operator $\bg$ is defined by
}
\begin{equation}
  \label{dgscheme}
  \int_{\mathbb{T}} \bg({\vec \phi_h}) \cdot \vec \psi_h \d x
  =
  -
  \int_{\mathbb{T}}{\vec f}({\vec \phi_h}) \cdot \partial_x {\vec \psi_h} \d x
  +
  \sum_{i=0}^{M-1} {\vec F}({\vec \phi_h}(x_i^-),{\vec \phi_h}(x_i^+)) \cdot \jump{\vec \psi_h}_i
\end{equation}
{
for all $\vec\phi_h,\vec\psi_h\in\mathbb V_h^p$, where
$\vec F:U\times U\rightarrow\reals^m$ is a numerical flux. For
$i=0,\ldots,M-1$, define the one-sided traces by
\[
\vec\psi_h(x_i^\pm):=\lim_{s\searrow0}\vec\psi_h(x_i\pm s)
\]
and set
\[
\jump{\vec\psi_h}_i
:=\vec\psi_h(x_i^-)-\vec\psi_h(x_i^+),
\qquad
\avg{\vec\psi_h}_i
:=\frac12\left(\vec\psi_h(x_i^-)+\vec\psi_h(x_i^+)\right).
\]
}

{
The analysis in what follows is quite general, and no specific assumptions are imposed on the structure of the discrete diffusion operator $\bD$. For example, one may use the symmetric interior penalty method
\cite{ArnoldBrezziCockburnMarini:2001, Di-PietroErn:2012}, defined by
}
\begin{align}\nonumber
    -\int_{\mathbb{T}} \bD(\vec \phi_h) \cdot \vec \psi_h \d x
    &= \int_{\mathbb{T}} \qp{\vec A(\vec \phi_h) \partial_x \vec \phi_h} \cdot \partial_x \vec \psi_h{\d x}\nonumber\\
    &\quad - \sum_{i=0}^{M-1} \Bigg(
    \jump{\vec \psi_h}_i \cdot \avg{\vec A(\vec \phi_h) \partial_x \vec \phi_h}_i
    \nonumber\\
    &\qquad\qquad{+
    \jump{\vec \phi_h}_i \cdot \avg{\Transpose{\vec A(\vec \phi_h)} \partial_x \vec \psi_h}_i}
    \nonumber\\
    &\qquad\qquad-
    {\frac{1}{h_i}
    \left(\boldsymbol{\Sigma}_i(\vec\phi_h)\jump{\vec \phi_h}_i\right)\cdot
    \jump{\vec \psi_h}_i}\Bigg),\nonumber
\end{align}
{
{
where $\boldsymbol{\Sigma}_i(\vec\phi_h)\in\reals^{m\times m}$ is a symmetric positive semidefinite \emph{penalty matrix} at $x_i$, chosen using $\vec A(\vec\phi_h(x_i^-))$ and $\vec A(\vec\phi_h(x_i^+))$ so that it is sufficiently large on the diffusive components. For constant diffusion matrices, the same fixed matrix may be used at every interface.
}
{For systems with partial diffusion, $\boldsymbol{\Sigma}_i(\vec\phi_h)$ vanishes on the components on which diffusion does not act.}
}

{
To account for the discontinuous nature of the dG approximation, we define a spatial dG energy seminorm and its $\leb{2}$-in-time counterpart.

\begin{definition}[dG energy seminorms]\label{def:mesh-norms}
  For $j=0,\ldots,M-1$, let $I_j=(x_j,x_{j+1})$. For a function $\vec v$ satisfying
  \[
    \vec v|_{I_j}\in\sobh{1}(I_j,\reals^m)
    \qquad\text{for }j=0,\ldots,M-1,
  \]
  we define the spatial dG energy seminorm by
  \begin{equation}\nonumber
    \enorm{\vec v}_{\mathrm{dG}}^2
    :=
    \sum_{j=0}^{M-1}
    \qp{\Norm{\partial_x \vec v}_{\leb{2}(I_j)}^2 + h_j^{-1}|\jump{\vec v}_j|^2}.
  \end{equation}
  For a time-dependent function $\vec w$ satisfying this spatial regularity for almost every $s\in(0,T)$, we define for $t\in[0,T]$:
  \begin{equation}\nonumber
    \enorm{\vec w}_{\leb{2}(0,t;\mathrm{dG})}^2
    :=
    \int_0^t \enorm{\vec w(s)}_{\mathrm{dG}}^2\d s.
  \end{equation}
  We refer to $\enorm{\vec w}_{\leb{2}(0,t;\mathrm{dG})}$ as the $\leb{2}$-in-time dG
  energy seminorm.
\end{definition}
}

We consider Runge-Kutta (RK) temporal discretizations. For convection-dominated or hyperbolic problems,
either explicit RK methods or implicit-explicit (IMEX) RK methods are suitable.
{
We partition the time interval $[0,T]$ into $N$ consecutive subintervals
with endpoints $0 = t_0 < t_1 < \dots < t_N = T$. The size of the $n$-th time step is
$\tau_n = t_n - t_{n-1}$, and the maximal time step size is $\tau:= \max_{1\leq n\leq N} \tau_n$.
}

{
The lowest-order IMEX scheme in our framework is the forward-backward Euler method,
also known as ARS(1,1,1). Its explicit and implicit Butcher tableaux are, respectively,
}
\begin{equation}\nonumber
  \begin{array}{c|cc}
    0   & 0   & 0  \\
    1 & 1 & 0  \\
    \hline
    & 1   & 0  \\
  \end{array}
  \qquad
  \begin{array}{c|cc}
    0   & 0   & 0  \\
    1 & 0 & 1  \\
    \hline
    & 0   & 1  \\
  \end{array}
  .
\end{equation}
{
If we apply this IMEX scheme to \eqref{eq:sdisc} with initial datum
$\vec u_h^0 := \vec u_h(0)$, the fully discrete problem is to find $\vec u_h^n$, $n=1,\ldots,N$, satisfying
}
{
\begin{equation}\nonumber
  \frac{\vec u_h^n - \vec u_h^{n-1}}{\tau_n} = -\bg ({\vec u}_h^{n-1}) + \veps \bD (\vec u_h^n).
\end{equation}
}
{
Our analysis is not restricted to a specific scheme but accommodates general RK and multistep methods. Any such method applied to \eqref{eq:sdisc} yields approximate solutions $\vec u_h^n\in\mathbb V_h^p$ at the time points $t_n$, $n=0,\ldots,N$.
}

\subsection{Reconstruction techniques}

{
The discrete solutions obtained from the RKdG scheme are piecewise polynomial in space (possibly discontinuous at element interfaces) and are defined only at the time points $t_n$. To apply the stability analysis in \S \ref{sec:gref}, we construct space-time reconstructions.
}

{
Our reconstruction proceeds in two steps, first in time and then in space. The reason for this ordering is that the temporal part of the PDE system \eqref{eq:cd-intro} is linear, whereas the spatial part is, in general, nonlinear.
}

{
Given the discrete approximations $\vec u_h^n\in\mathbb V_h^p$, $n=0,\ldots,N$, we first construct a temporal reconstruction $\widehat{\vec u}_h^{\mathrm t}$, which is Lipschitz continuous and piecewise polynomial in time and remains piecewise polynomial (possibly discontinuous) in space. Subsequently, we use $\widehat{\vec u}_h^{\mathrm t}$ to construct a space-time reconstruction $\widehat{\vec u}^{\mathrm{ts}}$ that is Lipschitz continuous in both space and time. Figure \ref{fig:reconstruction} illustrates this two-step procedure.
}

\begin{figure}[!ht]
  \begin{tikzpicture}[scale=0.75]
    \pgfplotstableread{
    x plot1   plot2   plot3
    0         0         0         0
    0.1111    0.3420    0.1248    0.1020
    0.2222    0.6428    0.2822    0.2122
    0.3333    0.8660    0.5011    0.3094
    0.4444    0.9848    0.7550    0.4029
    0.5556    0.9848    0.9656    0.5766
    0.6667    0.8660    1.0343    0.8716
    0.7778    0.6428    0.8885    1.0856
    0.8889    0.3420    0.5205    0.8406
    1.0000    0.0000         0         0
    }\dummydata
    \begin{axis}[
      domain=0:1,
      samples y=0, %ytick={1,...,4},
      zmin=0,
      xticklabels={,,},
      yticklabels={,,},
      zticklabels={,,},
      xlabel={$x$},
      ylabel={$t$},
      zlabel={$\vec u_h^n$},
      area plot/.style={
      fill opacity=0.25,
      draw=blue!80!black,thick,
      fill=blue,
      mark=none,
      }
      ]
      \pgfplotsinvokeforeach{1,...,3}{
      \addplot3 [area plot] table [x=x, y expr=#1, z=plot#1]
      {\dummydata};
      }
    \end{axis}
  \end{tikzpicture}
  \begin{tikzpicture}[scale=0.75]
    \begin{axis}[
      domain=0:1,
      samples y=0, %ytick={1,...,4},
      zmin=0,
      xticklabels={,,},
      yticklabels={,,},
      zticklabels={,,},
      xlabel={$x$},
      ylabel={$t$},
      zlabel={$\vec{\widehat u}_h^{\mathrm t}$},
      area plot/.style={
      draw=blue!80!black,thick
      }
      ]
      \addplot3[surf,  fill opacity=0.25,
      shader=faceted,
      point meta=explicit] file {plot2.dat};
    \end{axis}
  \end{tikzpicture}
  \begin{tikzpicture}[scale=0.75]
    \begin{axis}[
      domain=0:1,
      samples y=0, %ytick={1,...,4},
      zmin=0,
      xticklabels={,,},
      yticklabels={,,},
      zticklabels={,,},
      xlabel={$x$},
      ylabel={$t$},
      zlabel={$\vec{\widehat u}^{\mathrm{ts}}$},
      area plot/.style={
      draw=black,thick
      }
      ]
      \addplot3[surf, fill opacity=0.25,
      shader=faceted,
      point meta=explicit] file {plot3.dat};
    \end{axis}
  \end{tikzpicture}
  \caption{Two-stage reconstruction methodology. (Left) Discrete solution $\vec{u}^n_h$ at time nodes. (Center) Temporal reconstruction $\widehat{\vec u}_h^{\mathrm t}$. (Right) Space-time reconstruction $\widehat{\vec u}^{\mathrm{ts}}$.}
  \label{fig:reconstruction}
\end{figure}

{
For an arbitrary vector space $V$ and $\ell\in\mathbb N_0$, we define the space of $V$-valued piecewise
polynomials of degree at most $\ell$ in time as
\begin{equation}\label{tspace}
  \mathbb{V}_{\tau}^{\ell}(0,T; V)
  :=
  \{ v : [0,T] \rightarrow V \, :
  \,v|_{ (t_n,t_{n+1})} \in \mathbb{P}_{\ell}((t_n,t_{n+1}),V)
  \text{ for } n=0,\ldots,N-1\}.
\end{equation}
The subscript $\tau$ refers to the temporal partition, and $\ell$ denotes the polynomial degree in time.
}
{
The temporal reconstruction $\vut$ is $C^0$, or even $C^1$, in time and piecewise polynomial in space, with its polynomial degree in time chosen to match the order of accuracy of the time discretization method. For clarity of exposition, we present the construction for third-order schemes. The general construction for arbitrary polynomial degrees is described in \cite{DG16}.
}

{
\begin{definition}[Temporal reconstruction for third-order schemes]\label{def:grec}
  The temporal reconstruction $\vut$ is the unique element of $\mathbb{V}_{\tau}^{3}(0,T; \mathbb{V}_h^p)$ such that, for $n=0,\ldots,N-1$ and $j=n,n+1$,
  \begin{align}\label{grec}
      \vut|_{[t_n,t_{n+1}]}(t_j)
      &=
      \vec u_h^j,\nonumber\\
      \partial_t \vut|_{[t_n,t_{n+1}]}(t_j)
      &=
      -\bg({\vec u}_h^j)
      +
      \veps\bD({\vec u}_h^j).
  \end{align}
\end{definition}
}

{
\begin{lemma}[Properties of the temporal reconstruction \cite{DG16}]
  \label{lem:time-reconst}
  The temporal reconstruction $\vut$ defined in Definition \ref{def:grec} is well-defined and computable, and it belongs to $\operatorname{W}^{1,\infty}(0,T; \mathbb{V}_h^p)$.
\end{lemma}
}

{As in \cite[Assumption 17]{DG16}, we assume that the temporal reconstruction $\vut$ takes values in a compact and convex subset of $U$.}

{
Having constructed a temporal reconstruction, we next perform a spatial reconstruction to obtain a sufficiently smooth space-time reconstruction $\widehat{\vec u}^{\mathrm{ts}}$. Our approach is based on the spatial reconstruction developed for hyperbolic conservation laws in \cite{GMP_15}. To this end, we impose assumptions on the form of the numerical flux $\vec F$ introduced in \eqref{dgscheme}.
}

\begin{assumption}[Condition on the numerical flux]\label{ass1}
  {
  We assume that there exists a locally Lipschitz continuous function ${\vec w}: U \times U \rightarrow U$ such that, for every compact $K \subset U$, there exists a constant $C_w(K)>0$ satisfying
  \begin{equation}\label{w-cond}
    |{\vec w}({\vec a}, {\vec b}) - {\vec a} | + |{\vec w}({\vec a}, {\vec b}) - {\vec b} | \leq 
    C_w(K)|{\vec b} - {\vec a}| \qquad \text{for all } {\vec a}, {\vec b} \in K.
  \end{equation}
  The numerical flux ${\vec F}$ is assumed to have one of the following two forms for all ${\vec a}, {\vec b} \in U$:
  \begin{itemize}
    \item[(i)]
    \[
    {\vec F}({\vec a}, {\vec b}) ={\vec f}({\vec w}({\vec a}, {\vec b})).
    \]
    \item [(ii)]
    \[
    {\vec F}({\vec a}, {\vec b}) ={\vec f}({\vec w}({\vec a}, {\vec b}))  - \mu({\vec a},{\vec b};h) 
    h^\nu ({\vec b} - {\vec a}),
    \]
    {
    for some $\nu \in \mathbb{N}_0$ and some matrix-valued function $\mu$ which has the property that, for every compact $K \subset U$, there exists a constant $\mu_K > 0$ such that
    \[
    |\mu({\vec a},{\vec b};h)| \leq \mu_K\left(1+\frac{|{\vec a}-{\vec b}|}{h}\right)
    \]
    for all ${\vec a},{\vec b}\in K$ and for $h$ small enough.
    }
  \end{itemize}
  }
\end{assumption}

\begin{remark}[Admissible numerical fluxes] \label{rem:fluxes}
  The rather abstract conditions given in Assumption \ref{ass1} are satisfied in practice by various well-known numerical fluxes. For
  example, the Lax-Wendroff and Richtmyer numerical fluxes defined by
  \begin{equation}\label{def:LW}
    {\vec F}({\vec a},{\vec b})= {\vec f}({\vec w}({\vec a},{\vec b})), 
    \quad {\vec w}({\vec a},{\vec b})= \frac{{\vec a} + {\vec b}}{2} - \frac{\lambda}{2} ( {\vec f}({\vec b}) - {\vec f}({\vec a})),
  \end{equation}
  satisfy Assumption \ref{ass1} (i).

  The Lax-Friedrichs flux defined by
  \begin{equation}\label{def:LLF}
    {\vec F}({\vec a}, {\vec b}) =\frac{1}{2} \Big(
    {\vec f}({\vec a}) + {\vec f}({\vec b}) \Big)  - \lambda ({\vec b} - {\vec a}){,}
  \end{equation}
  satisfies Assumption \ref{ass1} (ii) with
  ${\vec w}({\vec a},{\vec b})=\frac{1}{2}({\vec a}+{\vec b})$, $\nu=0$, and
  \begin{equation*}
    \mu({\vec a},{\vec b};h)=\lambda \mathbb{I} - \frac{{\vec f}({\vec a}) - 2{\vec f}({\vec w}({\vec a},{\vec b})) + {\vec f}({\vec b})}
    {2\norm{{\vec b}-{\vec a}}^2} \otimes ({\vec b}-{\vec a})~,
  \end{equation*}
  where $\mathbb{I}$ denotes the $m \times m$ identity matrix. 
  
  {
  For a detailed discussion of the conditions imposed in Assumption
  \ref{ass1}, we refer the reader to \cite[Remark 21]{DG16}.
  }
\end{remark}

The space-time reconstruction $\widehat{\vec u}^{\mathrm{ts}}$ is obtained
by applying a spatial reconstruction operator to the temporal reconstruction
{$\widehat{\vec u}_h^{\mathrm t}(t,\cdot)$} at each time $t$.
This reconstruction incorporates the numerical flux structure
{via its interface values, which are determined by the function ${\vec w}$ defined in Assumption \ref{ass1}.}

\begin{definition}[Space-time reconstruction]\label{def:str}
  {
  For each fixed $t\in [0,T]$, let $\vut(t,\cdot)$ be the temporal reconstruction defined in Definition \ref{def:grec}.
  The space-time reconstruction $\strec(t,\cdot)$ is the unique element of $\mathbb{V}_h^{p+1}$ satisfying
  }
  \begin{align}\label{srec}
      \int_{\W} (\widehat{\vec u}^{\mathrm{ts}}(t,\cdot) - \vut(t,\cdot)) \cdot {\vec \psi_h}\,{\d x} &= 0
      {\quad \text{for all } {\vec \psi_h} \in \mathbb{V}_h^{p-1},}\nonumber\\
      \widehat{\vec u}^{\mathrm{ts}}(t, x_k^\pm)&= {\vec w}(\vut(t,x_k^-),\vut(t,x_k^+))
      {\quad \text{for } k=0,\ldots,M-1.}
  \end{align}
\end{definition}

\begin{remark}[Reconstruction at physical boundaries]
  {The definition above is written for periodic
    boundary conditions.  On a bounded interval with physical
    boundaries, the same reconstruction is used at interior
    interfaces. At a Dirichlet boundary, the corresponding endpoint
    value of $\strec$ is prescribed to be the boundary datum $\vec
    g_D$. For an inviscid problem, the prescribed boundary value is
    used on the incoming characteristic components, while the outgoing
    components are assigned the corresponding one-sided trace of the
    temporal reconstruction $\vut$. Thus, in the nonperiodic stability
    estimates below, the space-time reconstruction has the same
    prescribed boundary data as the exact solution. These boundary
    modifications remain local and do not alter the well-posedness or
    regularity properties of the reconstruction.}
\end{remark}

\begin{lemma}[Properties of {the} space-time reconstruction {\cite[Lemma 24]{DG16}}] \label{lem:space-time-reconst}
  The space-time reconstruction $\strec$ {defined in Definition \ref{def:str}} is well-defined {and locally computable, and it belongs to} $\operatorname{W}^{1,\infty}(0,T; {\mathbb{V}_h^{p+1}} \cap C(\W, \reals^m))$. Moreover, $\strec$ is Lipschitz continuous in space.
\end{lemma}

\begin{remark}[Regularity of {the} space-time reconstruction]
  Note that the spatial reconstruction is{, in general,} a \emph{nonlinear} operation
  and{ that} the temporal regularity of the space-time
  reconstruction $\strec$ depends on $\vec w$. When $\vec w$
  is nonlinear, $\strec$ is {generally} not a polynomial in time, which
  makes its numerical representation  delicate. 
\end{remark}

%%%%%%%%%%%%%%%%%%%%%%%%%%%%%%%%%%%%%%%%%%%%%%%%%%%%%%%%%%%%%%%%%%%%%%%%%%%%%%%
\subsection{Decomposing the residual}\label{sec:dres}

{
We insert the space-time reconstruction $\strec$ into the continuous problem
\eqref{eq:cd-intro} and define the residual by
}
\begin{equation}\label{def_rec}
  \vec r
  :=
  {\partial_t \strec}
  +
  \partial_x \vec f(\strec)
  -
  \veps \partial_x (\vec A(\strec) \partial_x\strec ).
\end{equation}
{
The regularity of $\vec r$ depends on the presence of diffusion.
{
If $\veps = 0$, then $\vec r \in \leb{2}((0,T) \times \W,\reals^m)$, whereas for $\veps > 0$, the diffusion term generally implies only that $\vec r \in \leb{2}(0,T;\sobh{-1}(\W,\reals^m))$.
}
To handle problems without diffusion, with partial diffusion, or with diffusion in every component within a single framework, we decompose the residual as
}
\begin{equation}\label{eq:res-decomp}
  \vec r = \vec r_1 + \veps\vec r_2,
\end{equation}
{
where we refer to $\vec r_1$ and $\vec r_2$ as the hyperbolic and parabolic contributions to the residual, respectively.
{Although the residual components may depend on $\veps$ through the numerical solution and its reconstructions, we suppress this dependence in the notation and define them by}
}
\begin{equation}\label{eq:r1}
  \vec r_1 :=
  \partial_t \strec + \partial_x \vec f(\strec) - \veps \bD(\vut)
  \in
  \leb{2}((0,T) \times \W,\reals^m),
\end{equation}
and
{
\begin{equation}\label{eq:r2}
  \vec r_2 :=
  \qp{
  \bD(\vut)
  -
  \partial_x (\vec A(\strec) \partial_x\strec ) }
  \in
  \leb{2}(0,T;\sobh{-1}(\W,\reals^m)).
\end{equation}
}
{
When $\veps = 0$, we have $\vec r \equiv \vec r_1$, and the analysis reduces to the hyperbolic framework developed in \cite{DG16}.
}

\begin{remark}[Splitting the residual]
  {
  The decomposition of $\vec r$ is not unique.
  For example, $\bD(\vut)$ in the definitions of $\vec r_1$ and $\vec r_2$ may be replaced by a suitable extension of $\bD$ applied to $\strec$, provided that the resulting components satisfy the regularity and duality requirements of Lemma \ref{lem:greb}.
  We choose the splitting defined by equations \eqref{eq:r1} and \eqref{eq:r2} because it evaluates the discrete diffusion operator from the semidiscrete scheme at the temporal reconstruction and therefore requires no extension of $\bD$.
  Although the optimal scaling of the residual quantities obtained from this splitting has not been established theoretically, the numerical experiments in Section \ref{sec:num} show that, in the convection-dominated regime, they converge at the same rates as the discretization errors.
  }
\end{remark}

\begin{remark}[{Several space dimensions}]
  {
  For hyperbolic systems endowed with a strictly convex entropy, the relative
  entropy stability framework can also be formulated in several space
  dimensions. A two-dimensional spatial reconstruction on Cartesian meshes
  using tensor product polynomial spaces has been constructed in Section 5.3
  of \cite{2017Giesselmann} and shown in Lemma 5.3 to be locally computable
  and Lipschitz continuous. However, the reconstruction in
  \cite{2017Giesselmann} leads to residuals that converge to zero much more
  slowly than the exact error. Thus, the major obstacle is finding a suitable
  reconstruction for the two-dimensional case.
  }
\end{remark}

\section{{General relative entropy estimates}}
\label{sec:gref}

{Let $m\in\rN$ and let $U\subset\reals^m$ be open.}
We consider the general nonlinear convection-diffusion system
\begin{equation}\label{eq:cd}
  \partial_t \vec u + \partial_x \vec f(\vec u)
  =
  \veps  \partial_x (\vec A(\vec u) \partial_x \vec u),
\end{equation}
where $\vec f \in \cont{2}(U, \reals^m)$ and $\vec A \in \cont{1}(U, \reals^{m \times m})$.
{We compare a weak solution $\vec u$ of the system \eqref{eq:cd} with a strong solution $\hatu \in
\sob{1}{\infty}((0,T) \times \W,\reals^m)$ of the perturbed system}
\begin{equation}\label{eq:cdr}
  \partial_t \hatu + \partial_x \vec f(\hatu)
  =
  \veps \partial_x (\vec A(\hatu) \partial_x \hatu) + \vec r{,}
\end{equation} 
where the residual is decomposed as $\vec r = \vec r_1 + \veps\vec r_2$, with\linebreak[1]
$$\vec r_1 \in
\leb{1}(0,T; \leb{2}(\W, \reals^m)), \ \vec r_2 \in \leb{2}(0,T;
\sobh{-1}( \W, \reals^m)). $$
{
Note that the residual components $\vec r_1$ and $\vec r_2$ defined in \eqref{eq:r1} and \eqref{eq:r2} satisfy these regularity assumptions.
{The initial data enter the weak formulations through $\vec u_0$ and $\hatu(0,\cdot)$, whereas periodic, Dirichlet, or inflow boundary conditions are imposed in the individual stability results according to the problem.}
}

{We assume that there exists a compact convex set $K\subset U$ such that $\vec u(t,x),\hatu(t,x)\in K$ for almost every $(t,x)\in(0,T)\times\W$.}

{We assume that the system \eqref{eq:cd} admits an entropy/entropy-flux pair
$(\eta,q)$, with $\eta \in \cont{2}(U, \reals)$ strictly convex and $q \in \cont{1}(U, \reals)$, satisfying the compatibility relation}
\begin{equation}\label{def:efp}
  \D q = \D \eta \D \vec f{,}
\end{equation}
{
where $\D$ denotes differentiation with respect to the column vector
$\vec u=(u_1,\ldots,u_m)^T$. Thus, $\D\eta,\D q\in\reals^{1\times m}$
are row gradients, $\D\vec f\in\reals^{m\times m}$ is the Jacobian of
$\vec f$, and $\D^2\eta\in\reals^{m\times m}$ is the Hessian of $\eta$.
}
Note that {the} existence of $q$ implies the following commutative property:
{
\begin{equation}
  \label{eq:comm}
  \Transpose{(\D \vec f)} \D^2 \eta = \D^2\eta \D \vec f{.}
\end{equation}
}
Moreover, when $U$ is simply connected, this commutative property is
also a sufficient condition for the existence of {a function $q$} satisfying \eqref{def:efp}.

\begin{remark}[Entropy]
  {
  In one spatial dimension, the commutative property \eqref{eq:comm} imposes $\tfrac{1}{2}m(m-1)$ scalar equations on $\eta(u_1,\ldots,u_m)$. Thus, in the scalar case ($m=1$), every strictly convex function $\eta\in\cont{2}(U,\reals)$ gives rise to an entropy/entropy-flux pair by choosing $q$ such that $q'=\eta'f'$. For $m>1$, constructing an entropy/entropy-flux pair is more involved. Nevertheless, many physically relevant systems of the form \eqref{eq:cd} admit an entropy/entropy-flux pair \cite{Daf10}.
  }
\end{remark}

{For a sufficiently smooth $\vec u$, we multiply \eqref{eq:cd}
by $\D\eta(\vec u)$ and use \eqref{def:efp}, the chain rule, and the
product rule to obtain the identities}
{
\begin{align*}
  \D\eta(\vec u)\partial_t \vec u
  &=
  \partial_t \eta(\vec u),
  \\
  \D\eta(\vec u)\D \vec f(\vec u)\partial_x \vec u
  &=
  \D q(\vec u)\partial_x \vec u
  =
  \partial_x q(\vec u),
  \\
  \D\eta(\vec u)\partial_x \qpbig{\vec A(\vec u)\partial_x \vec u}
  &=
  \partial_x \qpbig{\D\eta(\vec u)\vec A(\vec u)\partial_x \vec u}
  -
  \Transpose{\partial_x \vec u}\D^2\eta(\vec u)\vec A(\vec u)\partial_x \vec u .
\end{align*}
}
{Thus, we obtain the entropy balance}
\begin{equation*}
  \partial_t \eta(\vec u)
  +
  \partial_x ( q(\vec u) - \veps \D \eta(\vec u) \vec A(\vec u) \partial_x \vec u)
  =
  - \veps \Transpose{\partial_x \vec u} \D^2 \eta(\vec u) \vec A(\vec u) \partial_x \vec u.
\end{equation*}
{Throughout the paper, we restrict our attention to weak solutions that satisfy
the entropy inequality}
\begin{equation}\label{eq:eniq}
  \partial_t \eta(\vec u)
  +
  \partial_x ( q(\vec u) - \veps \D \eta(\vec u) \vec A(\vec u) \partial_x \vec u)
  \leq
  -\veps \Transpose{\partial_x \vec u} \D^2 \eta(\vec u) \vec A(\vec u) \partial_x \vec u
\end{equation}
{in the weak sense. Such a restriction is standard in the
study of hyperbolic conservation laws, including the system
\eqref{eq:cd} with $\veps = 0$.}

\begin{definition}[Entropy solution]\label{def:entsol}
  A function $\vec u \in \leb{\infty}((0,T) \times
  \W, \reals^m)$ {satisfying
  $\veps \vec A(\vec u) \partial_x \vec u
  \in \leb{1}((0,T)\times \W, \reals^m)$}
  is called a \emph{weak solution} of
  \eqref{eq:cd} if it satisfies
  \begin{equation}\label{def:weak_sol} 
    \int_0^T \int_\W \partial_t \vec \phi \cdot \vec u + \partial_x \vec \phi \cdot (\vec f(\vec u) 
    - \veps \vec A(\vec u) \partial_x \vec u)\d x \d t
    + \int_\W \vec u_0 \cdot \vec \phi(0,\cdot) \d x=0
  \end{equation}
  for all $\vec \phi \in \cont{\infty}_0([0,T) \times \W, \reals^m).$
  We call a weak solution
  an \emph{entropy solution} of
  \eqref{eq:cd} with respect to an entropy/entropy-flux pair
  $(\eta,q)$ {if it has the additional regularity
  $\veps\partial_x(\D\eta(\vec u))\vec A(\vec u)\partial_x\vec u
  \in \sob{-1}{1}((0,T)\times\W)$ and satisfies}
  \begin{multline}\label{def:entr_sol}
    \int_0^T \int_\W \partial_t \Phi \eta(\vec u) + \partial_x \Phi (q(\vec u) 
    - \veps \D \eta(\vec u) \vec  A(\vec u) \partial_x \vec u)
    \d x \d t
    \\
    - {\left\langle
      \veps\partial_x(\D\eta(\vec u))\vec A(\vec u)\partial_x\vec u,
      \Phi
    \right\rangle}
    + \int_\W \Phi(0,\cdot) \eta(\vec u_0) \d x 
    \geq 0
  \end{multline}
  for all $\Phi \in \cont{\infty}_0([0,T) \times \W, \reals_+)${,}
  {where $\langle\cdot,\cdot\rangle$ denotes the distributional pairing.}
\end{definition}

{By approximation with smooth functions, we may also use
Lipschitz test functions satisfying the same support conditions in
\eqref{def:weak_sol} and \eqref{def:entr_sol}. The weak formulation of
\eqref{eq:cdr} is}
  \begin{equation}\label{weak_sol_res}
  \int_0^T \int_\W \partial_t \vec \phi \cdot \hatu + \partial_x \vec \phi \cdot(\vec f(\hatu) 
  - \veps \vec A(\hatu) \partial_x \hatu) + {\vec r_1 \cdot \vec \phi}\d x \d t
  {+\veps\langle\vec r_2,\vec\phi\rangle}
  + \int_\W \hatu(0,\cdot) \cdot\vec \phi(0,\cdot)\d x=0
\end{equation}
for all $\vec \phi \in \sob{1}{\infty}_0([0,T) \times \W,\reals^m).$
{Taking $\vec \phi = \Phi \Transpose{\D \eta(\hatu)}$ gives}
  \begin{multline}\label{entr_sol_res}
  \int_0^T \int_\W \partial_t \Phi \eta(\hatu) + \partial_x \Phi q(\hatu)  
  -{\partial_x (\Phi \D \eta(\hatu)) \veps \vec A(\hatu) \partial_x \hatu}
  +{\vec r_1\cdot(\Phi \Transpose{\D \eta(\hatu)})}  \d x \d t
  {+\veps\langle\vec r_2,\Phi\Transpose{\D\eta(\hatu)}\rangle}\\
  + \int_\W  \Phi(0,\cdot) \eta(\hatu(0,\cdot)) \d x=0.
\end{multline}
{Equation \eqref{entr_sol_res} is obtained by integrating by
parts in time and space and using the chain rule
$\partial_t\eta(\hatu)=\D\eta(\hatu)\partial_t\hatu$ and the compatibility
relation \eqref{def:efp}.}

\begin{remark}[Well-posedness]
  {
  The question of existence and uniqueness of entropy solutions to \eqref{eq:cd} is delicate and cannot be answered completely in general. For example, in scalar
  hyperbolic conservation laws with a non-convex flux, e.g., $f(u)=u^3$ with
  $\veps=0$, entropy solutions need not be unique when only one entropy
  inequality is prescribed for one entropy/entropy-flux pair
  \cite{JMS95}.
  Indeed, the stability estimates we obtain do not imply
  uniqueness of entropy solutions, but only weak-strong uniqueness, \ie entropy solutions are unique as long as a Lipschitz continuous entropy solution with the same initial data exists.
  }
\end{remark}

We now introduce the relative entropy framework for comparing solutions 
of \eqref{eq:cd} and \eqref{eq:cdr}.
This methodology, widely used in
fluid mechanics for Euler and Navier-Stokes-Fourier equations,
provides the stability estimates needed for our analysis.
\begin{definition}[Relative entropy]
  \label{def:relent}
  {
  For an entropy/entropy-flux pair $(\eta,q)$ with $\eta$ strictly convex, \emph{relative
  entropy} and \emph{relative entropy-flux} between two states $\vec u, \vec v
  \in U$ are defined as
  }
  \begin{equation}\nonumber
      \eta(\vec u| \vec v)
      :=
      \eta(\vec u) - \eta(\vec v) - \D \eta(\vec v)(\vec u - \vec v),\quad
      q(\vec u| \vec v)
      :=
      q(\vec u) - q(\vec v) - \D \eta(\vec v)(\vec f(\vec u) - \vec f(\vec v)).
  \end{equation}
\end{definition}

\begin{lemma}[Relative entropy inequality]
  \label{lem:greb}
  {
  Let $\vec u$ be an entropy solution of \eqref{eq:cd} with respect to the entropy/entropy-flux pair $(\eta,q)$, and let $\hatu \in \sob{1}{\infty}((0,T) \times \W, \reals^m)$ be a Lipschitz continuous solution of \eqref{eq:cdr}.
  }
  {
  {When $\veps>0$, we assume that}
  \begin{equation}\nonumber
    {\partial_x( \D^2 \eta(\hatu) (\vec u - \hatu))\cdot \vec A(\hatu) \partial_x \hatu \in \leb{1}([0,T) \times \W)}
  \end{equation}
  {and that, for every $\Phi \in \sob{1}{\infty}_0([0,T) \times \W)$, the duality pairing}
  \begin{equation*}
    {\left\langle
      \vec r_2,
      \Phi \D^2 \eta(\hatu) (\vec u - \hatu)
    \right\rangle}
  \end{equation*}
  {is well defined.}
  }
  Then, for any $\Phi \in  \sob{1}{\infty}_0([0,T) \times \W, \reals_+)$, the following relative entropy inequality holds:
  \begin{multline*}
    0 \leq \int_0^T \int_\W \bigg[ \partial_t \Phi \eta(\vec u| \hatu) 
      + \partial_x \Phi \Big( q(\vec u|\hatu) - \veps (\D \eta(\vec u) - \D \eta(\hatu)) \vec A(\vec u) \partial_x \vec u \\
      + \veps {(\vec A(\hatu) \partial_x \hatu)\cdot \D^2 \eta(\hatu)(\vec u - \hatu)} \Big) \bigg] \d x \d t \\
      - {\left\langle
        \veps \partial_x(\D\eta(\vec u))\vec A(\vec u)\partial_x\vec u,
        \Phi
      \right\rangle} \\
      + \int_0^T \int_\W \Phi \bigg[
      {\veps\partial_x(\D\eta(\hatu))\vec A(\vec u)\partial_x\vec u}
      + \veps \vec A(\hatu) \partial_x \hatu \cdot \partial_x( \D^2 \eta(\hatu) (\vec u - \hatu)) \\
      - \vec r_1 \cdot \D^2 \eta(\hatu) (\vec u - \hatu) - \partial_x(\D \eta(\hatu))\vec f(\vec u|\hatu) \bigg] \d x \d t \\
      - \veps\langle \vec r_2 , \Phi \D^2 \eta(\hatu) (\vec u - \hatu) \rangle + \int_\W \Phi(0,\cdot) \eta(\vec u_0| \hatu(0,\cdot))\d x{,}
  \end{multline*}
  {where}
  \begin{equation}\nonumber
      \vec f (\vec u| \hatu) := \vec f(\vec u) - \vec f(\hatu) - \D \vec f(\hatu) (\vec u- \hatu).
  \end{equation}
\end{lemma}

\begin{proof}
  {
  For $\Phi \in \sob{1}{\infty}_0([0,T) \times \W, \reals_+)$, we use
  $\Phi$ in \eqref{def:entr_sol}, $-\Phi$ in \eqref{entr_sol_res},
  $\vec\phi=-\Phi\Transpose{\D\eta(\hatu)}$ in \eqref{def:weak_sol}, and
  $\vec\phi=\Phi\Transpose{\D\eta(\hatu)}$ in \eqref{weak_sol_res}.
  We then add the three identities to the entropy inequality, apply the chain
  rule, and use \eqref{eq:cdr} to replace $\partial_t\hatu$. By \eqref{eq:comm},
  \begin{align*}
    &-\partial_x\bigl(\D\eta(\hatu)\bigr)
    \bigl(\vec f(\vec u)-\vec f(\hatu)\bigr)
    +\bigl(\D\vec f(\hatu)\partial_x\hatu\bigr)
    \cdot\D^2\eta(\hatu)(\vec u-\hatu)\\
    &\qquad
    =-\partial_x\bigl(\D\eta(\hatu)\bigr)\vec f(\vec u|\hatu).
  \end{align*}
  {The entropy dissipation term involving $\vec u$ is understood
  in the distributional sense through the pairing in
  \eqref{def:entr_sol}. The compactness of $K$ and the regularity of
  $\hatu$ give}
  \begin{equation*}
    {\partial_x\bigl(\D\eta(\hatu)\bigr)
    \in\leb{\infty}((0,T)\times\W).}
  \end{equation*}
  {Since
  $\veps\vec A(\vec u)\partial_x\vec u
  \in\leb{1}((0,T)\times\W,\reals^m)$,}
  \begin{equation*}
    {
    \veps\partial_x\bigl(\D\eta(\hatu)\bigr)
    \vec A(\vec u)\partial_x\vec u
    \in\leb{1}((0,T)\times\W).}
  \end{equation*}
  {The term involving
  $\partial_x(\vec A(\hatu)\partial_x\hatu)$ is integrated by parts in space. The
  integrability assumption and the vanishing spatial trace of $\Phi$ give}
  \begin{align*}
    &-\veps\int_0^T\int_\W
    \Phi\D^2\eta(\hatu)(\vec u-\hatu)
    \cdot\partial_x\bigl(\vec A(\hatu)\partial_x\hatu\bigr)\d x\d t\\
    &\quad=
    \veps\int_0^T\int_\W
    \partial_x\Phi\,
    \bigl(\vec A(\hatu)\partial_x\hatu\bigr)
    \cdot\D^2\eta(\hatu)(\vec u-\hatu)\d x\d t\\
    &\qquad+
    \veps\int_0^T\int_\W
    \Phi\,
    \bigl(\vec A(\hatu)\partial_x\hatu\bigr)
    \cdot\partial_x\bigl(\D^2\eta(\hatu)(\vec u-\hatu)\bigr)\d x\d t.
  \end{align*}
  Using these identities, we obtain the stated inequality.
  }
\end{proof}

\begin{remark}[{Well-definedness of the duality pairing with $\vec r_2$}]
  {If the matrix $\D^2\eta(\vec u)\vec A(\vec u)$ is uniformly positive definite on the relevant range of states, the entropy dissipation provides an $\sobh{1}$ estimate for every component of $\vec u-\hatu$, and the duality pairing in Lemma \ref{lem:greb} is therefore well defined.}
  {However, without uniform positive definiteness, the entropy dissipation need not provide an $\sobh{1}$ estimate for every component.}
  {Nevertheless, the duality pairing remains well defined if $\D^2 \eta(\hatu)$ is diagonal and, for every component of $\vec r_2$ assumed only in $\sobh{-1}$, an $\sobh{1}$ estimate is available from diffusion for the same component of $\vec u-\hatu$.}
\end{remark}

%------------------------------------------------------------------------------------------------------------------------------------------------------------

\section{Linear scalar problems}

\label{sec:ls}

{We consider} the scalar case $m=1$ with linear flux $f(u) = b u$ for {a} constant $b \in \reals$ and diffusion $A(u) \equiv 1$.
Thus, equations \eqref{eq:cd} and \eqref{eq:cdr} become
\begin{equation}\label{eq:cdsl}
  \partial_t u + b \partial_x u = \veps \del_{xx} u
\end{equation}
and
\begin{equation}\label{eq:cdrsl}
  \partial_t \widehat u + b \partial_x \widehat u = \veps \del_{xx} \widehat u + r.
\end{equation}
We assume a splitting of the residual $r = r_1 + \veps r_2$ as in \eqref{eq:cdr}, with $r_1 \in \leb{1}(0,T; \leb{2}(\W))$ and $r_2 \in \leb{2}(0,T; \sobh{-1}(\W))$.

\begin{remark}[{Regularity assumptions}]
  \label{rem:reg}
  {Entropy solutions are assumed to satisfy $u\in \leb{\infty}((0,T)\times\W)$, with the additional regularity $u\in \leb{2}(0,T;\sobh{1}(\W))$ in the viscous case $\veps>0$.}
\end{remark}

\begin{remark}[{Boundary conditions}]
  \label{rem:bcsr}
  {For $\veps>0$, Dirichlet boundary data are prescribed on $(0,T)\times\partial\W$. For $\veps=0$, boundary data are prescribed only at the inflow boundary, namely the left boundary when $b>0$ and the right boundary when $b<0$. In both cases, the prescribed boundary data are denoted by $g_D$.}
\end{remark}

\begin{remark}[Entropy]\label{rem:ent-sl}
  For any $\eta \in \cont{2}(\reals, \reals)$, we may define an entropy flux 
  $q(u):=b \eta(u)$. {The quadratic entropy/entropy-flux pair}
  {
  \begin{equation}\label{eq:ent-sl}
    \eta(u)=\frac12u^2,\qquad q(u)=\frac b2u^2
  \end{equation}
  yields the relative entropy/entropy-flux pair}
  \begin{equation}\label{def:resl}
    \eta(u|\widehat u) = \frac{1}{2} (u - \widehat u)^2,\quad
    q(u|\widehat u) = \frac{b}{2} (u - \widehat u)^2.
  \end{equation}
  {This choice yields an $\leb{2}$-estimate for the error. In the inviscid case $\veps=0$, Kruzhkov entropy solutions \cite{Kru70} satisfy \eqref{eq:eniq} for the quadratic entropy \eqref{eq:ent-sl}.}
\end{remark}

\begin{theorem}[{Stability estimate for the viscous problem}]
  \label{the:sl}
  {
	  For $\veps>0$, let $u \in \leb{2}(0,T;\sobh{1}(\W))
	  \cap \leb{\infty}((0,T)\times\W)$ be an entropy solution of \eqref{eq:cdsl} with respect to the
	  entropy/entropy-flux pair \eqref{eq:ent-sl} and with initial data
	  {$u_0\in\leb{2}(\W)$}.
	  Let $\widehat u\in\sob{1}{\infty}((0,T)\times\W)$ be a solution of
	  \eqref{eq:cdrsl} with initial data
	  $\widehat u_0\in\sob{1}{\infty}(\W)$. Assume that $u$ and
	  $\widehat u$ satisfy the Dirichlet boundary condition with the same data
	  $g_D\in\sob{1}{\infty}((0,T)\times\partial\W)$.
	  Then, for almost every $t\in(0,T)$,}
  \begin{multline*}
    \frac 14 \Norm{u - \widehat u}_{\leb{\infty}(0,t; \leb{2}( \W))}^2
    +
    \frac \veps2\Norm{\partial_x (u - \widehat u)}_{\leb{2}((0,t)\times \W)}^2 \\
    \leq
    \Norm{u_0- \widehat u_0}_{\leb{2}(\W)}^2     
    +
    4\Norm{ r_1}_{\leb{1}(0,t; \leb{2}( \W))}^2
    +
    \veps \Norm{r_2}_{\leb{2}(0,t; \sobh{-1}(\W))}^2{.}
  \end{multline*}
\end{theorem}

\begin{remark}[Compatibility conditions]
  {
  Theorem \ref{the:sl} implicitly assumes that the initial data
  $\widehat u_0$ and the prescribed boundary data are compatible at
  $t=0$, in the sense that
  $\widehat u_0|_{\partial\W}=g_D(0,\cdot)$.
  }
\end{remark}

\begin{proof}[Proof of Theorem \ref{the:sl}]
  {Set $e:=u-\widehat u$.} Since
  \begin{equation}\nonumber
      \D \eta(u) - \D \eta (\widehat u) = \D^2 \eta(\widehat u)( u- \widehat u) = e \quad \text{and} \quad
      f(u) - f(\widehat u) - \D f(\widehat u) (u- \widehat u) =0,
  \end{equation}
  {substituting \eqref{def:resl} into the relative entropy inequality in Lemma \ref{lem:greb} gives}
  \begin{multline}\label{slreb_hyp}
      0 \leq  \int_0^T \int_\W
      \partial_t \Phi \frac 1 2 e^2 + \partial_x \Phi {\left(\frac b 2 e^2
      -\veps e  \partial_x  e
      \right)}\\
      + \Phi {\left( - \veps   (\partial_x  e )^2
      -  r_1 e
      \right)}\d x \d t
      - \veps\langle r_2 , \Phi  e \rangle
      + \int_\W
      \Phi(0,\cdot)\frac 12 e(0,\cdot)^2 \d x.
  \end{multline} 
  
  {
  To adapt the cutoff functions in
  \cite[Eqns (5.3.11) and (5.3.12)]{Daf10} to the bounded interval $\W$,
  fix $0<s<T$ and
  $0<\delta<\min\{\frac12,T-s\}$. We use the test function
  $\Phi_\delta(x,t):=\psi_\delta(x)\zeta_\delta(t)$, where
  }
  \begin{equation}\label{def:psidelta}
    \begin{split}
      & \psi_\delta(x)= \left\{ \begin{array}{ccc}
        \frac x \delta &:& 0 \leq x < \delta\\
        1  &:& \delta \leq x < 1- \delta\\
        \frac {1-x} \delta &:& 1- \delta \leq x \leq 1\\
      \end{array},\right. \qquad
      \zeta_\delta(t)= \left\{ \begin{array}{ccc}
        1  &:& t< s\\
        1- \frac {t-s} \delta &:& s \leq t \leq s+ \delta\\
        0 &:& t > s+ \delta
      \end{array}\right..
    \end{split}
  \end{equation}
  
  {
  Substituting $\Phi_\delta$ into \eqref{slreb_hyp}, we obtain
	  \begin{multline}\label{slreb}
	      0 \leq  - \frac{1}{\delta}\int_s^{s+\delta} \int_\W
	      \frac 12 e^2  \d x \d t
	      + \frac 1 \delta \int_0^s \int_0^\delta {\left(\frac b2 e^2 -\veps e ( \partial_x  e)\right)}\d x \d t\\
	      - \frac 1 \delta \int_0^s \int_{1-\delta}^1 {\left(\frac b2e^2 -\veps e  (\partial_x e)\right)}\d x \d t
	      + {\int_0^T \int_\W \Phi_\delta\Big(  - \veps  ( \partial_x  e)^2
	      - r_1 e
	      \Big)}\d x \d t\\
	      - \veps\langle r_2 , \Phi_\delta  e\rangle
	      + {\int_\W \psi_\delta \frac 12 e(0,\cdot)^2 \d x  + R_\delta,}
	  \end{multline}
	  {
	  where
	  \begin{align*}
	    R_\delta
	    &:=
	    \frac{1}{\delta}\int_s^{s+\delta}
	    \left[
	    \int_0^\delta (1-\psi_\delta)\frac12 e^2 \d x
	    +
	    \int_{1-\delta}^1(1-\psi_\delta)\frac12 e^2 \d x
	    \right]\d t \\
	    &\quad+
	    \frac{1}{\delta}\int_s^{s+\delta}\zeta_\delta(t)
	    \left[
	    \int_0^\delta
	    \left(\frac b2 e^2-\veps e\partial_xe\right)\d x
	    -
	    \int_{1-\delta}^1
	    \left(\frac b2 e^2-\veps e\partial_xe\right)\d x
	    \right]\d t .
		  \end{align*}
		  Since $e(t,\cdot)\in\sobh{1}_0(\W)$ for almost every $t$, its trace vanishes at both endpoints. Thus, for $x \in (0,\delta)$,
		  \[
		    e(t,x)=\int_0^x \partial_x e(t,y)\d y,
		  \]
		  and the analogous identity holds near $x=1$. The Cauchy--Schwarz inequality therefore gives, for $I_\delta=(0,\delta)$ or $I_\delta=(1-\delta,1)$,
		  \[
		    \int_{I_\delta} e(t,x)^2\d x
		    \leq
		    \delta^2\int_{I_\delta}|\partial_x e(t,x)|^2\d x.
		  \]
		  Consequently,
		  \begin{align*}
		    \frac{1}{\delta}\int_s^{s+\delta}\int_{I_\delta}e^2\d x\d t
		    &\leq
		    \delta\int_s^{s+\delta}\int_{I_\delta}|\partial_x e|^2\d x\d t
		    \to0,\\
		    \frac{1}{\delta}\int_s^{s+\delta}\int_{I_\delta}|e\partial_x e|\d x\d t
		    &\leq
		    \int_s^{s+\delta}\int_{I_\delta}|\partial_x e|^2\d x\d t
		    \to0.
		  \end{align*}
		  Both limits follow from the absolute continuity of the integral of $|\partial_x e|^2$.
		  Hence $R_\delta\to0$ as $\delta\searrow0$.
		  Since $t\mapsto\frac12\int_\W e(t,x)^2\d x$ belongs to $\leb{1}(0,T)$, the Lebesgue differentiation theorem gives, for almost every $s\in(0,T)$,
		  \[
		    \frac1\delta\int_s^{s+\delta}\int_\W
		    \frac12e^2\d x\d t
		    \longrightarrow
		    \int_\W\frac12e(s,x)^2\d x.
		  \]
	  Let $\mathbf{1}_{(0,s)}$ denote the characteristic function of $(0,s)$.
	  Since $0\leq \Phi_\delta\leq1$, $\Phi_\delta\to\mathbf{1}_{(0,s)}$ pointwise almost everywhere, and $-\veps(\partial_x e)^2-r_1e \in \leb{1}((0,T)\times\W)$, dominated convergence gives
	  \[
	    \int_0^T\int_\W
	    \Phi_\delta\left(-\veps(\partial_x e)^2-r_1e\right)\d x\d t
		    \longrightarrow
		    \int_0^s\int_\W
		    \left(-\veps(\partial_x e)^2-r_1e\right)\d x\d t.
		  \]
		  Since $0\leq\psi_\delta\leq1$ and $\psi_\delta\to1$ pointwise almost everywhere, dominated convergence also gives
		  \[
		    \int_\W \psi_\delta\frac12 e(0,\cdot)^2\d x
		    \longrightarrow
		    \int_\W \frac12 e(0,\cdot)^2\d x.
	  \]
	  }
	  }
	  {
		  The two spatial cutoff terms in \eqref{slreb} arise from $\partial_x\psi_\delta$. For $I_\delta=(0,\delta)$ or $I_\delta=(1-\delta,1)$, applying \eqref{boundarylemma} at the relevant endpoint with $\zeta=e$ and $\xi=\partial_xe$ gives
		  \[
		    \frac1\delta\int_0^s\int_{I_\delta}e^2\d x\d t
		    =
		    \delta\left(
		    \frac1{\delta^2}\int_0^s\int_{I_\delta}e^2\d x\d t
		    \right)
		    \to0,
		  \]
		  and
		  \[
		    \left|
		    \frac1\delta\int_0^s\int_{I_\delta}
		    e\partial_xe\d x\d t
		    \right|
		    \to0.
		  \]
		  Hence both spatial cutoff terms in \eqref{slreb} tend to zero.
	  For the duality pairing, $\Phi_\delta e\in \leb{2}(0,T;\sobh{1}_0(\W))$, and
	  \[
	    \partial_x(\Phi_\delta e)
	    =
	    \zeta_\delta\psi_\delta\partial_xe
	    +
	    \zeta_\delta(\partial_x\psi_\delta)e.
	  \]
		  The first term converges to $\mathbf 1_{(0,s)}\partial_xe$ in $\leb{2}((0,T)\times\W)$ by dominated convergence. For the second term,
		  \begin{align*}
	    \Norm{\zeta_\delta(\partial_x\psi_\delta)e}_{\leb{2}((0,T)\times\W)}^2
	    &\leq
	    \frac1{\delta^2}\int_0^{s+\delta}
	    \left(
	      \int_0^\delta e^2\d x
	      +
	      \int_{1-\delta}^1e^2\d x
	    \right)\d t\\
	    &\leq
	    \int_0^{s+\delta}
	    \left(
	      \int_0^\delta|\partial_xe|^2\d x
	      +
	      \int_{1-\delta}^1|\partial_xe|^2\d x
		    \right)\d t
		    \to0.
		  \end{align*}
		  The last limit follows from the absolute continuity of the integral of $|\partial_xe|^2$. Poincar\'e's inequality therefore gives
		  \[
		    \Phi_\delta e
		    \longrightarrow
		    \mathbf 1_{(0,s)}e
		    \quad\text{in }\leb{2}(0,T;\sobh{1}_0(\W)).
		  \]
		  Consequently,
		  \begin{align*}
		    \langle r_2,\Phi_\delta e\rangle
		    &\longrightarrow
		    \langle r_2,\mathbf{1}_{(0,s)}e\rangle,\\
		    \left|\langle r_2,\mathbf{1}_{(0,s)}e\rangle\right|
		    &\leq
		    \Norm{r_2}_{\leb{2}(0,s;\sobh{-1}(\W))}
		    \Norm{\partial_xe}_{\leb{2}((0,s)\times\W)}.
		  \end{align*}
	  }
  
  {
	  Passing to the limit and applying H\"older's and Young's inequalities gives, for almost every $s\in(0,T)$,
	  \begin{multline*}
	      \frac12\Norm{e(s)}_{\leb{2}(\W)}^2
	      +
      \frac{\veps}{2}\Norm{\partial_xe}_{\leb{2}((0,s)\times\W)}^2
      \leq
      \frac12\Norm{e(0)}_{\leb{2}(\W)}^2
      +
      2\Norm{r_1}_{\leb{1}(0,s;\leb{2}(\W))}^2
      \\
      +
	      \frac18\Norm{e}_{\leb{\infty}(0,s;\leb{2}(\W))}^2
	      +
	      \frac{\veps}{2}\Norm{r_2}_{\leb{2}(0,s;\sobh{-1}(\W))}^2.
	  \end{multline*}
	  Since both terms on the left are nonnegative, each is bounded separately by the right-hand side. Taking the essential supremum over $s\in(0,t)$ in the estimate for $\Norm{e(s)}_{\leb{2}(\W)}^2$, using the monotonicity of the norms on the right, and adding the estimate for $\Norm{\partial_xe}_{\leb{2}((0,t)\times\W)}^2$ gives the stated result after rearrangement.
	  }
  
\end{proof}

{For the inviscid problem, we use the same relative entropy argument as in Theorem \ref{the:sl}, but state the stability estimate separately because only inflow boundary data are prescribed and the boundary terms require a different treatment.}

\begin{theorem}[{Stability estimate for the inviscid problem}]
  \label{cor:sl}
  {
  For $\veps=0$, let $u\in\leb{\infty}((0,T)\times\W)$ be an entropy
  solution of \eqref{eq:cdsl} with respect to the entropy/entropy-flux pair
  \eqref{eq:ent-sl} and with initial data $u_0\in\leb{\infty}(\W)$.
  Let $\widehat u\in\sob{1}{\infty}((0,T)\times\W)$ be a solution of
  \eqref{eq:cdrsl} with initial data
  $\widehat u_0\in\sob{1}{\infty}(\W)$.
  Assume that the same Lipschitz continuous data $g_D$ are prescribed for
  $u$ and $\widehat u$ at the inflow boundary.
  Then, for almost every $t\in(0,T)$,}
  \begin{equation}\label{eq:resl}
    \frac{1}{4}\Norm{u - \widehat u}_{\leb{\infty}(0,t;\leb{2}(\W))}^2
    \leq
    \Norm{u_0 - \widehat u_0}_{\leb{2}(\W)}^2
    + \Norm{r_1}_{\leb{1}(0,t; \leb{2}(\W))}^2.
  \end{equation}
\end{theorem}	

\begin{proof}
	  {
	  {Set $e:=u-\widehat u$. Applying Lemma \ref{lem:greb} with $\veps=0$, the entropy/entropy-flux pair \eqref{eq:ent-sl}, and the test function $\Phi_\delta$ defined in \eqref{def:psidelta} gives}
	  \begin{multline}\label{slreb_1}
      0 \leq  - \frac{1}{\delta}\int_s^{s+\delta} \int_\W
      \frac 12 e^2  \d x \d t
      + \frac 1 \delta \int_0^s \int_0^\delta \frac b2 e^2 \d x \d t\\
      - \frac 1 \delta \int_0^s \int_{1-\delta}^1 \frac b2 e^2 \d x \d t
      {- \int_0^T \int_\W
      \Phi_\delta r_1 e
      \d x \d t
      + \int_\W 
      \psi_\delta \frac 12 e(0,\cdot)^2 \d x + R_\delta,}
  \end{multline}
  {
  where
  \begin{equation*}
    \begin{aligned}
      R_\delta
      =&
      \frac{1}{\delta}\int_s^{s+\delta}
      \left[
      \int_0^\delta (1-\psi_\delta)\frac12 e^2\d x
      +
      \int_{1-\delta}^1 (1-\psi_\delta)\frac12 e^2\d x
      \right]\d t\\
      &+
      \frac{1}{\delta}\int_s^{s+\delta}\zeta_\delta
      \left[
      \int_0^\delta \frac b2 e^2\d x
      -
      \int_{1-\delta}^1 \frac b2 e^2\d x
      \right]\d t.
    \end{aligned}
  \end{equation*}
  Since $e\in \leb{\infty}((0,T)\times\W)$,
  \[
    |R_\delta|
    \leq
    C\delta \Norm{e}_{\leb{\infty}((0,T)\times\W)}^2
    \longrightarrow0.
  \]
  Since $t\mapsto\frac12\int_\W e(t,x)^2\d x$ belongs to
  $\leb{1}(0,T)$, the Lebesgue differentiation theorem gives, for almost
  every $s\in(0,T)$,
  \[
    \frac1\delta\int_s^{s+\delta}\int_\W
    \frac12e^2\d x\d t
    \longrightarrow
    \int_\W\frac12e(s,x)^2\d x.
  \]
  Since $0\leq\Phi_\delta\leq1$,
  $\Phi_\delta\to\mathbf{1}_{(0,s)}$ pointwise almost everywhere, and
  $r_1e\in\leb{1}((0,T)\times\W)$, the dominated convergence theorem gives
  \[
    -\int_0^T\int_\W \Phi_\delta r_1e\d x\d t
    \to
    -\int_0^s\int_\W r_1e\d x\d t.
  \]
  Likewise, since $0\leq\psi_\delta\leq1$, $\psi_\delta\to1$ pointwise
  almost everywhere, and $e(0,\cdot)^2\in\leb{1}(\W)$, the same theorem gives
  \[
    \int_\W\psi_\delta\frac12 e(0,\cdot)^2\d x
    \to
    \int_\W\frac12 e(0,\cdot)^2\d x.
  \]
  }
  
  {
  If $b=0$, the boundary terms in \eqref{slreb_1} vanish and the estimate
  follows directly. We therefore assume $b>0$. The case $b<0$ is analogous
  with the inflow boundary at $x=1$. For sufficiently small $\delta$, if
  $t\geq\delta/b$ and $0<x<\delta$, the backward characteristic reaches
  $x=0$, so
  \[
    u(t,x)=g_D(t-x/b).
  \]
  Since $\widehat u(t,0)=u(t,0)=g_D(t)$, the Lipschitz continuity of $g_D$
  and $\widehat u$ gives
  \[
    |e(t,x)|
    \leq
    |g_D(t-x/b)-g_D(t)|
    +
    |\widehat u(t,x)-\widehat u(t,0)|
    \leq Lx,
  \]
  where $L$ is independent of $\delta$. The boundary contribution at $x=1$
  in \eqref{slreb_1} is nonpositive. At $x=0$,
  \[
    \begin{aligned}
      \frac1\delta\int_0^s\int_0^\delta\frac b2e^2\d x\d t
      &\leq
      \frac{\delta}{2}
      \Norm{e}_{\leb{\infty}((0,T)\times\W)}^2
      +
      \frac{bL^2s}{6}\delta^2
      \longrightarrow0.
    \end{aligned}
  \]
  Letting $\delta\searrow0$ in \eqref{slreb_1} therefore gives, for almost
  every $s\in(0,T)$,
  \[
    \frac12\Norm{e(s)}_{\leb{2}(\W)}^2
    \leq
    \frac12\Norm{e(0)}_{\leb{2}(\W)}^2
    -
    \int_0^s\int_\W r_1e\d x\d t.
  \]
	  Taking the essential supremum over $s\in(0,t)$ and applying H\"older's and
	  Young's inequalities gives \eqref{eq:resl}.
	  }
	  }
\end{proof}

\begin{remark}[{Other boundary conditions}]
  {
  For $\veps>0$, analogous estimates are obtained under periodic boundary
  conditions and other boundary conditions for which the boundary terms
  vanish. For $\veps=0$, an additional boundary
  flux term appears on the right-hand side of \eqref{eq:resl} when $u$ and
  $\widehat u$ have different inflow data. This term can be estimated if the
  boundary data belong to $\leb{2}((0,T)\times\partial\W)$.
  }
\end{remark}

We now apply Theorems \ref{the:sl} and \ref{cor:sl} to derive a posteriori error bounds for the fully discrete RKdG approximation.

\begin{corollary}[{Fully discrete a posteriori error bound}]
  \label{cor:lsapost}
  {
  Suppose $\{u_h^i\}_{i=0}^N\subset\mathbb{V}_h^p$ is an RKdG approximation with an
  (IMEX) RK temporal discretization of order at most 3, temporal
  reconstruction $\ut$ (Definition \ref{def:grec}), space-time reconstruction
  $\uts$ (Definition \ref{def:str}), and residuals $r_1$ and $r_2$ defined by
  \eqref{eq:r1}--\eqref{eq:r2}. Assume that $u$ satisfies the assumptions
  imposed on the entropy solution in Theorem \ref{the:sl} when $\veps>0$ and
  Theorem \ref{cor:sl} when $\veps=0$, and that $\uts$ satisfies the same
  boundary data as $u$.
  Then, for each $i=0,\dots,N$,}
  \begin{multline*}
    \Norm{u(t_i) - u_h^i}_{\leb{2}(\W)}^2
      + 2\veps \enorm{u - \ut}_{\leb{2}(0,t_i;{\mathrm{dG}})}^2 \\
      \leq 8 \Big(\Norm{{u_0} - \uts(0)}_{\leb{2}(\W)}^2
      + 4\Norm{r_1}_{\leb{1}(0,t_i; \leb{2}(\W))}^2
      + \veps \Norm{r_2}_{\leb{2}(0,t_i; \sobh{-1}(\W))}^2\Big) \\
      + 2\qp{\Norm{\uts(t_i) - u_h^i}_{\leb{2}(\W)}^2
      + 2\veps \enorm{\uts - \ut}_{\leb{2}(0,t_i;{\mathrm{dG}})}^2}
  \end{multline*}
\end{corollary}

\begin{remark}[{Interpretation of the error terms}]
  {
  The endpoint error at $t_i$ is evaluated using the numerical solution
  $u_h^i$. Since the dG energy seminorm is integrated in time, it is evaluated
  using the temporal reconstruction $\widehat u_h^{\mathrm t}$.
  }
\end{remark}

\begin{proof}[Proof of Corollary \ref{cor:lsapost}]
  {
  We decompose the discrete error using the space-time reconstruction $\uts$ as an intermediate function.
  The triangle inequality and $(a+b)^2 \leq 2(a^2 + b^2)$ give
  \begin{equation}\label{eq:decomp1}
    \Norm{u(t_i) - u_h^i}_{\leb{2}(\W)}^2 \leq 2\Norm{u(t_i) - \uts(t_i)}_{\leb{2}(\W)}^2 + 2\Norm{\uts(t_i) - u_h^i}_{\leb{2}(\W)}^2.
  \end{equation}
  Similarly, for the $\leb{2}$-in-time dG energy seminorm,
  \begin{equation}\label{eq:decomp2}
    \enorm{u - \ut}_{\leb{2}(0,t_i;{\mathrm{dG}})}^2 \leq 2\enorm{u - \uts}_{\leb{2}(0,t_i;{\mathrm{dG}})}^2 + 2\enorm{\uts - \ut}_{\leb{2}(0,t_i;{\mathrm{dG}})}^2.
  \end{equation}
  
  We next estimate the first terms on the right-hand sides. For $\veps>0$, both $u$ and $\uts$ are continuous across element interfaces, and hence
  $\enorm{u - \uts}_{\leb{2}(0,t_i;{\mathrm{dG}})}^2 = \Norm{\partial_x(u - \uts)}_{\leb{2}((0,t_i) \times \W)}^2$.
  Theorem \ref{the:sl} with $\widehat u = \uts$ therefore gives
  \begin{multline}\label{eq:cor-ls-reconstruction-bound}
    \Norm{u - \uts}_{\leb{\infty}(0,t_i; \leb{2}(\W))}^2 + 2\veps \enorm{u - \uts}_{\leb{2}(0,t_i;{\mathrm{dG}})}^2 \\
    \leq 4\Big(\Norm{u(0) - \uts(0)}_{\leb{2}(\W)}^2 + 4\Norm{r_1}_{\leb{1}(0,t_i; \leb{2}(\W))}^2 + \veps\Norm{r_2}_{\leb{2}(0,t_i; \sobh{-1}(\W))}^2\Big).
  \end{multline}

  For $\veps=0$, Theorem \ref{cor:sl} also implies
  \eqref{eq:cor-ls-reconstruction-bound} after the terms multiplied by $\veps$
  are omitted. Combining \eqref{eq:cor-ls-reconstruction-bound} with
  \eqref{eq:decomp1} and \eqref{eq:decomp2} gives the stated estimate.
  }
\end{proof}

%--------------------------------------------------------------------------------------------------------------------------------------
\section{Nonlinear scalar problems}
\label{sec:nls}

{We consider} the scalar case with nonlinear flux $f \in \cont{2}(\reals,\reals)$
and diffusion coefficient $A\in \cont{2}(\reals,\reals_+)$. 
{We assume that $A$ is uniformly positive. Let
$a_0:=\inf_{u\in\reals}A(u)>0$, and define $\widetilde{A}:=A/a_0$ and
$\widetilde{\veps}:=a_0\veps$. Since
$\widetilde{\veps}\widetilde{A}(u)=\veps A(u)$ and
$\inf_{u\in\reals}\widetilde{A}(u)=1$, the equation is unchanged, and after
dropping the tildes we assume without loss of generality that
$\inf_{u\in\reals}A(u)=1$.}
Thus, equations \eqref{eq:cd} and \eqref{eq:cdr} become
\begin{equation}\label{eq:cdsn}
  \partial_t u + \partial_x f(u)
  =
  \veps \del_{x} (A(u)\partial_x u)
\end{equation}
and
\begin{equation}\label{eq:cdrsn}
  \partial_t \widehat u + \partial_x f(\widehat u)
  =
  \veps \del_{x} (A(\widehat u)\partial_x \widehat u)+ r
\end{equation}
where the residual term decomposes as $r = r_1 + \veps r_2$, with $r_1 \in \leb{1}(0,T; \leb{2}(\W))$ and $r_2 \in \leb{2}(0,T; \sobh{-1}(\W))$.

\begin{remark}[Entropy]
  \label{rem:ent-cdsn}
  We use the entropy/entropy-flux pair
  {
  \begin{equation}\label{eq:ent-cdsn}
    \eta(u) = \frac{1}{2} u^2, \quad q(u)= \int^u_0 \tilde u f'(\tilde u) \operatorname{d} \tilde u.
  \end{equation}
  }
  {The quadratic entropy yields an $\leb{2}$ estimate for
  $u-\widehat u$. In the inviscid case $\veps=0$, Kruzhkov entropy solutions
  satisfy \eqref{eq:eniq} for every convex entropy, including the quadratic
  entropy in \eqref{eq:ent-cdsn}.}
\end{remark}

\begin{remark}[Boundary conditions]
  {
  For $\veps>0$, the estimates are derived under Dirichlet boundary conditions
  and are also valid under periodic boundary conditions. For $\veps=0$, the
  classification of the boundary as inflow or outflow depends on the solution,
  and Dirichlet data must therefore be interpreted weakly
  \cite{Ott96,KL01}. We restrict the inviscid analysis to periodic boundary
  conditions on $\mathbb{T}$.
  }
\end{remark}

\begin{theorem}[{Stability estimate for the viscous problem}]\label{the:sn}
  {
  For $\veps>0$, let $u\in\leb{2}(0,T;\sobh{1}(\W))
  \cap\leb{\infty}((0,T)\times\W)$ be an entropy solution of
  \eqref{eq:cdsn} with respect to the
  entropy/entropy-flux pair \eqref{eq:ent-cdsn} and with initial data
  {$u_0\in\leb{2}(\W)$}.
  Let $\widehat u\in\sob{1}{\infty}((0,T)\times\W)$ be a solution of
  \eqref{eq:cdrsn} with initial data
  $\widehat u_0\in\sob{1}{\infty}(\W)$.
  Assume that both $u$ and $\widehat u$ satisfy the Dirichlet boundary
  condition with data
  $g_D\in\sob{1}{\infty}((0,T)\times\partial\W)$.
  Then, for almost every $t\in(0,T)$,}
  \begin{multline}\nonumber
      \Norm{u - \widehat u}_{\leb{\infty}(0,t;\leb{2}(\W))}^2  
      +2\veps\Norm{\partial_x u - \partial_x \widehat u}_{\leb{2}((0,t)\times \W)}^2\\
      \leq \Big(4\Norm{u_0 - \widehat u_0}_{\leb{2}(\W)}^2
      + 16\Norm{r_1}_{\leb{1}(0,t;\leb{2}( \W))}^2
      + 8\veps \Norm{r_2}_{\leb{2}(0,t; \sobh{-1}(\W))}^2\Big) \exp{\qp{8{\Lambda[\widehat u]}t}},
  \end{multline}
  where
  \begin{equation}\nonumber
    {\Lambda[\widehat u]} := \Big(\veps \Norm{\partial_x \widehat u}_{\leb{\infty}((0,T)\times \W)}
    \sup_{v \in {K}} |A'(v)|^2 +
    \sup_{v \in {K}} |f''(v)|\Big)
    \Norm{\partial_x \widehat u}_{\leb{\infty}((0,T)\times \W)}{.}
  \end{equation}
\end{theorem}

\begin{proof}[Proof of Theorem \ref{the:sn}]
  {
  Set $e:=u-\widehat u$. For the entropy in \eqref{eq:ent-cdsn},
  \[
    \eta(u|\widehat u)=\frac12e^2,
    \qquad
    \D\eta(u)-\D\eta(\widehat u)=\D^2\eta(\widehat u)e=e.
  \]
  Applying Lemma \ref{lem:greb} gives
  }
  \begin{multline}\label{sn-reb}
      0 \leq  \int_0^T \int_\W
      \partial_t \Phi \frac 12 e^2 + \partial_x \Phi {\Big(q(u|\widehat u)
      -\veps e ( A(u) \partial_x u
      - A(\widehat u) \partial_x \widehat u) \Big)}\\
      + \Phi{\Big(  - \veps \partial_x e
      ( A(u) \partial_x u  - A(\widehat u) \partial_x \widehat u)
      -  r_1 e
      - \partial_x  \widehat u f( u|\widehat u) 
      \Big)}\d x \d t
      \\ 
      - \veps\langle r_2 , \Phi e\rangle
      + \int_\W 
      \Phi(0,\cdot) \frac 12e(0, \cdot)^2 \d x.
  \end{multline}
  {
  {
  Fix $0<s<T$ and
  $0<\delta<\min\{\frac12,T-s\}$, and use the test function $\Phi_\delta$
  from \eqref{def:psidelta} in \eqref{sn-reb}. Since $u$ and $\widehat u$
  have the same Dirichlet boundary data,
  $e\in\leb{2}(0,T;\sobh{1}_0(\W))$.
  Since $u,\widehat u\in K$, there is a constant $C>0$ such that
  \[
  \begin{aligned}
    |q(u|\widehat u)|+|f(u|\widehat u)|
    &\leq C|e|^2,\\
    A(u)\partial_xu-A(\widehat u)\partial_x\widehat u
    &=
    A(u)\partial_xe
    +(A(u)-A(\widehat u))\partial_x\widehat u
    \in\leb{2}((0,T)\times\W).
  \end{aligned}
  \]
  The term multiplied by $\Phi_\delta$ is therefore integrable. By Lemma
  \ref{lem:boundary-est} and the cutoff estimates established after
  \eqref{slreb}, with $\zeta=e$ and
  \[
    \xi=A(u)\partial_xu-A(\widehat u)\partial_x\widehat u,
  \]
  the limit $\delta\searrow0$ may be taken in \eqref{sn-reb}. For almost
  every $s\in(0,T)$,
  \[
    \left|\langle r_2,\mathbf1_{(0,s)}e\rangle\right|
    \leq
    \Norm{r_2}_{\leb{2}(0,s;\sobh{-1}(\W))}
    \Norm{\partial_xe}_{\leb{2}((0,s)\times\W)}.
  \]
	}

  Using $\inf_{u\in\reals}A(u)=1$ and $u,\widehat u\in K$, we have
  \[
  \begin{aligned}
    \veps\partial_xe
    (A(u)\partial_xu-A(\widehat u)\partial_x\widehat u)
    &\geq
    \veps|\partial_xe|^2
    -\veps
    \sup_{v\in{K}}|A'(v)|
    |e|\,|\partial_xe|\,|\partial_x\widehat u|,\\
    |f(u|\widehat u)|
    &\leq
    \frac12\sup_{v\in{K}}|f''(v)|\,|e|^2.
  \end{aligned}
  \]
  Passing to the limit and applying H\"older's and Young's inequalities gives,
  for almost every $s\in(0,T)$,
  \begin{equation}\label{sn-reb-int}
  \begin{aligned}
    \frac12\Norm{e(s)}_{\leb{2}(\W)}^2
    +\frac{\veps}{2}
    \Norm{\partial_xe}_{\leb{2}((0,s)\times\W)}^2
    \leq{}&
    \frac12\Norm{e(0)}_{\leb{2}(\W)}^2
    +2\Norm{r_1}_{\leb{1}(0,s;\leb{2}(\W))}^2\\
    &+\veps\Norm{r_2}_{\leb{2}(0,s;\sobh{-1}(\W))}^2
    +\frac18\Norm{e}_{\leb{\infty}(0,s;\leb{2}(\W))}^2
    +{\Lambda[\widehat u]}
    \Norm{e}_{\leb{2}((0,s)\times\W)}^2.
  \end{aligned}
  \end{equation}
  Both terms on the left-hand side of \eqref{sn-reb-int} are nonnegative.
  Taking the essential supremum over $s\in(0,t)$ in the estimate for
  $\Norm{e(s)}_{\leb{2}(\W)}^2$, using that the norms on the right-hand side
  are nondecreasing in $s$, and adding the result to the estimate for
  $\Norm{\partial_xe}_{\leb{2}((0,t)\times\W)}^2$ gives, after rearrangement,
  \begin{equation}\label{sn-reb-int3}
  \begin{aligned}
    \frac14\Norm{e}_{\leb{\infty}(0,t;\leb{2}(\W))}^2
    +\frac{\veps}{2}
    \Norm{\partial_xe}_{\leb{2}((0,t)\times\W)}^2
    \leq{}&
    \Norm{e(0)}_{\leb{2}(\W)}^2
    +4\Norm{r_1}_{\leb{1}(0,t;\leb{2}(\W))}^2\\
    &+2\veps\Norm{r_2}_{\leb{2}(0,t;\sobh{-1}(\W))}^2
    +2{\Lambda[\widehat u]}
    \Norm{e}_{\leb{2}((0,t)\times\W)}^2.
  \end{aligned}
  \end{equation}
  Gronwall's inequality then yields the stated estimate.
  }
\end{proof}

\begin{theorem}[{Stability estimate for the inviscid problem}]
  \label{cor:sn}
  {
  For $\veps=0$, let
  $u\in\leb{\infty}((0,T)\times\mathbb{T})$ be an entropy solution of
  \eqref{eq:cdsn} with respect to the entropy/entropy-flux pair
  \eqref{eq:ent-cdsn} and with initial data
  $u_0\in\leb{\infty}(\mathbb{T})$.
  Let $\widehat u\in\sob{1}{\infty}((0,T)\times\mathbb{T})$ be a solution
  of \eqref{eq:cdrsn} with initial data
  $\widehat u_0\in\sob{1}{\infty}(\mathbb{T})$.
  Assume that both $u$ and $\widehat u$ satisfy periodic boundary conditions.
  Then, for almost every $t\in(0,T)$,}
  \begin{equation*}
    {\Norm{u - \widehat u}_{\leb{\infty}(0,t;\leb{2}(\mathbb{T}))}^2}
    \leq \Big(2{\Norm{u_0 - \widehat u_0}_{\leb{2}(\mathbb{T})}^2}
    + 4\Norm{r_1}_{\leb{1}(0,t;\leb{2}(\mathbb{T}))}^2\Big) \exp\qp{4{\Lambda[\widehat u]}t}{,}
  \end{equation*}
  where
  \begin{equation*}
    {\Lambda[\widehat u]}:=
    \sup_{v \in {K}} |f''(v)|
    {\norm{\widehat u}_{\leb{\infty}(0,T;\sob{1}{\infty}(\mathbb{T}))}}{,}
  \end{equation*}
  
\end{theorem}
\begin{proof}
  {Set $e := u - \widehat u$ and choose the spatially constant test function
  \[
    \Phi_\delta(x,t) = \zeta_\delta(t),
  \]
  where $\zeta_\delta$ is the temporal cutoff from \eqref{def:psidelta}. Since
  $\partial_x\Phi_\delta=0$, applying Lemma \ref{lem:greb} with $\veps=0$
  and this test function gives}
  \begin{equation}\label{cor:sn-simple}
    0 \leq -\frac{1}{\delta} \int_s^{s+\delta} \int_{\mathbb{T}} \frac{1}{2} e^2 \d x \d t
    + {\int_0^T \int_{\mathbb{T}} \zeta_\delta(t)
    \bigg(-r_1 e - \partial_x \widehat u \cdot f(u|\widehat u)\bigg) \d x \d t}
    + \int_{\mathbb{T}} \frac{1}{2} e(0,\cdot)^2 \d x.
  \end{equation} 
  
  {Since $K$ is convex and $u,\widehat u\in K$, the relative flux satisfies}
  \begin{equation*}
    |f(u|\widehat u)| \leq \frac{1}{2}\sup_{v \in K} |f''(v)| \cdot e^2.
  \end{equation*}
  {The products $r_1e$ and
  $\partial_x\widehat u\,f(u|\widehat u)$ belong to
  $\leb{1}((0,T)\times\mathbb{T})$, and
  $t\mapsto\|e(t)\|_{\leb{2}(\mathbb{T})}^2$ belongs to $\leb{1}(0,T)$.
  Since $0\leq\zeta_\delta\leq1$ and
  $\zeta_\delta\to\mathbf{1}_{(0,s)}$ pointwise almost everywhere, the
  dominated convergence theorem gives the limits of the time integrals, and
  the Lebesgue differentiation theorem gives the limit of the first term.
  Hence, for almost every $s\in(0,T)$,}
  \begin{equation}\label{eq:gronwall-prep}
    \frac{1}{2}\|e(s)\|_{\leb{2}(\mathbb{T})}^2 \leq \frac{1}{2}\|e(0)\|_{\leb{2}(\mathbb{T})}^2
    + {\int_0^s \left|\int_{\mathbb{T}}r_1e\d x\right|\d t}
    + {\Lambda[\widehat u]} \int_0^s \|e(t)\|_{\leb{2}(\mathbb{T})}^2 dt.
  \end{equation}
  By Cauchy-Schwarz and Young's inequality,
  \begin{equation*}
    {\int_0^s \left|\int_{\mathbb{T}}r_1e\d x\right|\d t}
    \leq \|r_1\|_{\leb{1}(0,s; \leb{2}(\mathbb{T}))}^2 + \frac{1}{4}\|e\|_{\leb{\infty}(0,s; \leb{2}(\mathbb{T}))}^2.
  \end{equation*}
  
  {Since the right-hand side of \eqref{eq:gronwall-prep} is
  nondecreasing in $s$, taking the essential supremum over
  $s\in(0,t)$ and using the estimate for the residual term gives}
  \begin{equation*}
    \frac{1}{4}\|e\|_{\leb{\infty}(0,t; \leb{2}(\mathbb{T}))}^2 \leq \frac{1}{2}\|e(0)\|_{\leb{2}(\mathbb{T})}^2
    + \|r_1\|_{\leb{1}(0,t; \leb{2}(\mathbb{T}))}^2
    + {\Lambda[\widehat u]} \int_0^t \|e(s)\|_{\leb{2}(\mathbb{T})}^2 ds.
  \end{equation*}
  {Multiplying by 4 and applying Gronwall's inequality gives the stated
  estimate.}
\end{proof}

\begin{remark}[Extension to nonlinear systems]
  {
  For nonlinear systems endowed with a strictly convex entropy, estimates
  analogous to Theorem~\ref{the:sn} follow from Lemma~\ref{lem:greb} when
  $\D^2\eta(\vec u)\vec A(\vec u)$ is uniformly positive definite, with constants
  depending on upper and lower bounds for the eigenvalues of $\D^2\eta$; we omit
  the details of this argument. In the case where
  $\D^2\eta(\vec u)\vec A(\vec u)$ is only positive semidefinite, however, the
  analysis becomes more intricate, as illustrated by the examples with partial
  diffusion in Sections \ref{sec:lw} and \ref{sec:nw}.
  }
\end{remark}

\begin{corollary}[{Fully discrete a posteriori error bound}]
  \label{cor:nlsapost}
  {
  Suppose $\{u_h^i\}_{i=0}^N\subset\mathbb{V}_h^p$ is an RKdG approximation with an
  (IMEX) RK temporal discretization of order at most 3, temporal
  reconstruction $\ut$ (Definition \ref{def:grec}), space-time reconstruction
  $\uts$ (Definition \ref{def:str}), and residuals $r_1$ and $r_2$ defined by
  \eqref{eq:r1}--\eqref{eq:r2}. Assume that $u$ satisfies the assumptions
  imposed on the entropy solution in Theorem \ref{the:sn} when $\veps>0$ and
  Theorem \ref{cor:sn} when $\veps=0$, and that, when $\veps>0$, $\uts$
  satisfies the same Dirichlet boundary data as $u$.
  Then, for each $i=0,\dots,N$,}
  \begin{multline*}
      \Norm{u(t_i) - u_h^i}_{\leb{2}(\W)}^2
      +
      \veps \enorm{u - \ut}_{\leb{2}(0,t_i;{\mathrm{dG}})}^2 \\
      \leq
      2 \Big(4\Norm{{u_0} - \uts(0)}_{\leb{2}(\W)}^2
      + 16\Norm{r_1}_{\leb{1}(0,t_i;\leb{2}( \W))}^2
      + 8\veps \Norm{r_2}_{\leb{2}(0,t_i; \sobh{-1}(\W))}^2\Big) \exp{\qp{8  \Lambda[\uts]t_i}} \\
      +
      2\qp{\Norm{\uts(t_i) - u_h^i}_{\leb{2}(\W)}^2
      +
      \veps \enorm{\uts - \ut}_{\leb{2}(0,t_i;{\mathrm{dG}})}^2},
  \end{multline*}
  {where $\Lambda[\uts]$ is defined as in Theorem \ref{the:sn} when $\veps>0$ and as in Theorem \ref{cor:sn} when $\veps=0$.}
\end{corollary}
\begin{proof}
  {
  The stated estimate follows from the error splitting and triangle inequality
  in the proof of Corollary \ref{cor:lsapost} and from Theorems \ref{the:sn}
  and \ref{cor:sn} for $\veps>0$ and $\veps=0$, respectively, with
  $\widehat u=\uts$.
  }
\end{proof}

\section{Linear Systems}\label{sec:yl}

{
We consider a system for $\vec u\in\reals^m$ with linear flux
$\vec f(\vec u)=\vec B\vec u$, constant diffusion matrix
$\vec A(\vec u)=\vec A$, and diffusion parameter $\veps\geq0$, where
$\vec A,\vec B\in\reals^{m\times m}$.
We assume that the flux matrix $\vec B$ is hyperbolic, { meaning that it is
diagonalizable with real eigenvalues.}
After transforming to Riemann invariants, we may therefore take
$\vec B=\operatorname{diag}(b_1,\dots,b_m)$.
We order the diagonal entries so that the first $m_1$ are positive, the next
$m_2-m_1$ are zero, and the remaining entries are negative, where
$0\leq m_1\leq m_2\leq m$.
}
Thus, equations \eqref{eq:cd} and \eqref{eq:cdr} become
{
\begin{align}
  \partial_t \vec u + \vec B \partial_x \vec u
  &= \veps \vec A \del_{xx} \vec u,\label{eq:cdyl}\\
  \partial_t \hatu + \vec B \partial_x \hatu
  &= \veps \vec A \del_{xx} \hatu + \vec r,\label{eq:cdryl}
\end{align}
}
{
where the residual is decomposed as $\vec r = \vec r_1 + \veps\vec r_2$, with
$\vec r_1 \in \leb{1}(0,T;\allowbreak\leb{2}(\W;\reals^m))$ and
$\vec r_2 \in \leb{2}(0,T;\allowbreak\sobh{-1}(\W;\reals^m))$.
}

\begin{remark}[Entropy]\label{rem:ent-yl} 
  {
  For a symmetric positive definite matrix $\vec M$, consider the strictly convex quadratic entropy
  \begin{equation}\label{def:els}
    \eta(\vec u)=\frac{1}{2}\Transpose{\vec u}\vec M\vec u.
  \end{equation}
  The compatibility condition \eqref{eq:comm} and entropy flux are then given by
  \[
    \vec M\vec B=\Transpose{(\vec M\vec B)},
    \qquad
    q(\vec u)=\frac{1}{2}\Transpose{\vec u}\vec M\vec B\vec u.
  \]
  Such a matrix $\vec M$ exists because $\vec B$ is hyperbolic.
  }
\end{remark}

\begin{remark}[Compatibility with diffusion]
  {
  We assume that $\vec M$ can be chosen and normalized so that
  \[
    \Transpose{\vec\xi}\vec M\vec A\vec\xi\geq |\vec\xi|^2
    \qquad\text{for every }\vec\xi\in\reals^m.
  \]
  }
\end{remark}

\begin{theorem}[{Stability estimate for the viscous problem}]
  \label{the:yl}
  {
  For $\veps>0$, let
  $\vec u\in\leb{2}(0,T;\sobh{1}(\W,\reals^m))$
  be an entropy solution of \eqref{eq:cdyl} with respect to the entropy
  \eqref{def:els} and with initial data
  {$\vec u_0\in\leb{2}(\W,\reals^m)$}, and let
  $\hatu\in\sob{1}{\infty}((0,T)\times\W,\reals^m)$ be a solution of
  \eqref{eq:cdryl} with initial data
  $\hatu_0\in\sob{1}{\infty}(\W,\reals^m)$.
  Assume that both solutions satisfy the Dirichlet boundary condition with data
  $\vec g_D\in\sob{1}{\infty}((0,T)\times\partial\W,\reals^m)$.
  Then, for almost every $t\in(0,T)$,}
  {
  \begin{multline*}
    \frac 1 4 \Norm{\vec M^{1/2}(\vec u - \hatu)}^2_{\leb{\infty}(0,t;\leb{2}(\W))}
    +
    \frac{\veps}{2}
    \Norm{\partial_x \vec u - \partial_x \hatu}_{\leb{2}((0,t) \times \W)}^2
    \\
    \leq
    \Norm{\vec M^{1/2}(\vec u_0 - \hatu_0)}^2_{\leb{2}(\W)}
    +
    4\Norm{\vec M^{1/2}\vec r_1}_{\leb{1}(0,t;\leb{2}( \W))}^2
    +
    \veps \Norm{\vec M\vec r_2}_{\leb{2}(0,t; \sobh{-1}(\W))}^2.
  \end{multline*}
  }
\end{theorem}

\begin{proof}
  {
  Set $\vec e:=\vec u-\hatu$. The entropy pair \eqref{def:els} gives
  \begin{equation}\label{def:reyl} 
    \eta(\vec u|\hatu) =\frac{1}{2} \Transpose{\vec e}\vec M\vec e, \qquad
    q(\vec u|\hatu) =\frac{1}{2} \Transpose{\vec e}\vec M\vec B\vec e.
  \end{equation}
  {
  Fix $0<s<T$ and choose
  $0<\delta<\min\{\frac12,T-s\}$. Since $\D^2\eta=\vec M$ and $\vec f$ is
  linear, applying Lemma \ref{lem:greb} with the test function $\Phi_\delta$
  defined in \eqref{def:psidelta} gives
  }
	  }
	  {
	  \begin{multline}\label{reb-yl2}
	      0 \leq  - \frac 1 \delta \int_s^{s+ \delta} \int_\W \eta(\vec u| \hatu) \d x \d t
	       +\frac 1 \delta \int_0^{s} \int_0^{\delta} \qp{q(\vec u|\hatu)
	      -\veps \vec e \cdot \vec M\vec A \partial_x \vec e} \d x \d t\\
	      - \frac 1 \delta \int_0^{s} \int_{1- \delta}^1 \qp{q(\vec u|\hatu)
	      -\veps \vec e \cdot \vec M\vec A \partial_x \vec e} \d x \d t
	       +\int_0^T \int_\W \Phi_\delta  \qp{  - \veps \partial_x \vec e \cdot
	      \vec M\vec A \partial_x \vec e
	      - \vec r_1 \cdot \vec M\vec e} \d x \d t \\
	       - \veps\langle \vec r_2 , \Phi_\delta \vec M\vec e\rangle
	      + \int_\W
	      \psi_\delta \eta(\vec u_0| \hatu(0,\cdot))\d x + R_\delta
	  \end{multline}
	  }
	  {
		  {
		  We use the definition of $R_\delta$ given after \eqref{slreb}, with
			  $\frac12e^2$ and $\frac b2e^2-\veps e\partial_xe$ replaced by
			  $\eta(\vec u|\hatu)$ and
			  $q(\vec u|\hatu)-\veps\vec e\cdot\vec M\vec A\partial_x\vec e$,
			  respectively. Since $\vec u$ and $\hatu$ have the same Dirichlet
			  boundary data,
			  $\vec e\in\leb{2}(0,T;\sobh{1}_0(\W;\reals^m))$.
			  Moreover,
			  \[
			    \eta(\vec u|\hatu)\leq C|\vec e|^2,
		    \qquad
		    |q(\vec u|\hatu)|\leq C|\vec e|^2,
		    \qquad
			    |\vec e\cdot\vec M\vec A\partial_x\vec e|
			    \leq C|\vec e|\,|\partial_x\vec e|,
			  \]
			  where $C>0$ depends only on $\vec M$, $\vec B$, and $\vec A$.
			  Using $\vec e\in\leb{2}(0,T;\sobh{1}_0(\W;\reals^m))$ and the
			  preceding bounds on $\eta(\vec u|\hatu)$, $q(\vec u|\hatu)$, and
			  $\vec e\cdot\vec M\vec A\partial_x\vec e$, we apply the cutoff
			  argument following \eqref{slreb} componentwise. This gives
			  $R_\delta\to0$ and the limits of the other terms in
			  \eqref{reb-yl2} not involving $\vec r_2$. The same cutoff argument
			  also gives
		  \[
		    \Phi_\delta\vec e
		    \longrightarrow
		    \mathbf 1_{(0,s)}\vec e
		    \quad\text{in }\leb{2}(0,T;\sobh{1}_0(\W;\reals^m)).
		  \]
		  Since $\vec M$ is constant and symmetric,
		  \[
		  \begin{aligned}
		    \langle\vec r_2,\Phi_\delta\vec M\vec e\rangle
		    &=
		    \langle\vec M\vec r_2,\Phi_\delta\vec e\rangle
		    \longrightarrow
		    \langle\vec M\vec r_2,\mathbf1_{(0,s)}\vec e\rangle,\\
		    \left|
		    \langle\vec M\vec r_2,\mathbf1_{(0,s)}\vec e\rangle
		    \right|
		    &\leq
		    \Norm{\vec M\vec r_2}_{\leb{2}(0,s;\sobh{-1}(\W))}
		    \Norm{\partial_x\vec e}_{\leb{2}((0,s)\times\W)}.
		  \end{aligned}
		  \]
		  }
		  Using $\Transpose{\vec\xi}\vec M\vec A\vec\xi\geq|\vec\xi|^2$,
		  H\"older's inequality, and Young's inequality, we obtain, for almost
		  every $s\in(0,T)$,
	  \begin{equation}\label{reb-yl-int}
	  \begin{aligned}
	    \frac12\Norm{\vec M^{1/2}\vec e(s)}_{\leb{2}(\W)}^2
	    +\frac{\veps}{2}
	    \Norm{\partial_x\vec e}_{\leb{2}((0,s)\times\W)}^2
	    \leq{}&
	    \frac12\Norm{\vec M^{1/2}\vec e(0)}_{\leb{2}(\W)}^2\\
	    &+2\Norm{\vec M^{1/2}\vec r_1}_{\leb{1}(0,s;\leb{2}(\W))}^2\\
	    &+\frac18\Norm{\vec M^{1/2}\vec e}_{\leb{\infty}(0,s;\leb{2}(\W))}^2
	    +\frac{\veps}{2}
	    \Norm{\vec M\vec r_2}_{\leb{2}(0,s;\sobh{-1}(\W))}^2.
	  \end{aligned}
	  \end{equation}
			  Since both terms on the left-hand side of \eqref{reb-yl-int} are nonnegative, each is bounded by the right-hand side, which is nondecreasing in $s$. Taking the essential supremum over $s\in(0,t)$ in the bound for $\Norm{\vec M^{1/2}\vec e(s)}_{\leb{2}(\W)}^2$, adding the resulting inequality to the bound for $\Norm{\partial_x\vec e}_{\leb{2}((0,t)\times\W)}^2$, and rearranging gives the stated estimate.
	  }
\end{proof}

\begin{theorem}[{Stability estimate for the inviscid problem}]
  \label{cor:yl}
  {
  For $\veps=0$, let
  $\vec u\in\leb{\infty}((0,T)\times\W,\reals^m)$ be an entropy solution of
  \eqref{eq:cdyl} with respect to the entropy \eqref{def:els} and with initial
  data $\vec u_0\in\leb{\infty}(\W,\reals^m)$, and let
  $\hatu\in\sob{1}{\infty}((0,T)\times\W,\reals^m)$ be a solution of
  \eqref{eq:cdryl} with initial data
  $\hatu_0\in\sob{1}{\infty}(\W,\reals^m)$.
  Assume that the same Lipschitz continuous boundary data $\vec g_D$ are
  prescribed for both solutions along the incoming characteristics.
  Then, for almost every $t\in(0,T)$,}
  {
  \begin{equation*}
    \frac 1 4\Norm{\vec M^{1/2}(\vec u - \hatu)}^2_{\leb{\infty}(0,t;\leb{2}(\W))}
    \leq
    \frac 1 2 \Norm{\vec M^{1/2}(\vec u_0 - \hatu_0)}^2_{\leb{2}(\W)}
    +
    \Norm{\vec M^{1/2}\vec r_1}_{\leb{1}(0,t;\leb{2}( \W))}^2.
  \end{equation*}
  }
\end{theorem}
\begin{proof}
  {
  {
  Set $\vec e:=\vec u-\hatu$. Fix $0<s<T$ and choose
  $0<\delta<\min\{\frac12,T-s\}$.
  Applying Lemma \ref{lem:greb} with $\veps=0$ and the test function
  $\Phi_\delta$ gives \eqref{reb-yl2} with the diffusion terms omitted.
  The argument following \eqref{slreb_1} gives $R_\delta\to0$ and the
  limits of the terms in \eqref{reb-yl2} not involving
  $q(\vec u|\hatu)$.
  }
  It remains to establish the boundary inequalities
  \begin{equation}\label{linsyshypboundary1}
  \begin{aligned}
    \limsup_{\delta\searrow0}
    \frac1\delta\int_0^s\int_0^\delta
    q(\vec u|\hatu)\d x\d t&\leq0,\\
    \liminf_{\delta\searrow0}
    \frac1\delta\int_0^s\int_{1-\delta}^1
    q(\vec u|\hatu)\d x\d t&\geq0.
  \end{aligned}
  \end{equation}
  Since $\vec B$ is diagonal and $\vec M\vec B$ is symmetric,
  \[
    \vec M\vec B=\vec B\vec M,
    \qquad
    (b_i-b_j)M_{ij}=0.
  \]
  For each distinct characteristic speed $\lambda$, let
  $\mathcal I_\lambda:=\{i\in\{1,\ldots,m\}:b_i=\lambda\}$, and let
  $\vec e_\lambda$ and $\vec M_\lambda$ denote the subvector of $\vec e$ and
  the principal submatrix of $\vec M$ associated with $\mathcal I_\lambda$.
  Then
  \[
    q(\vec u|\hatu)
    =
    \underbrace{
    \frac12\sum_{\lambda>0}
    \lambda\,
    \Transpose{\vec e_\lambda}\vec M_\lambda\vec e_\lambda
    }_{q_1\geq0}
    +
    \underbrace{
    \frac12\sum_{\lambda<0}
    \lambda\,
    \Transpose{\vec e_\lambda}\vec M_\lambda\vec e_\lambda
    }_{q_2\leq0}.
  \]
  The zero-speed blocks make no contribution.
  Since $q_2\leq0\leq q_1$, the average near $x=0$ is bounded above by the
  average of $q_1$, and the average near $x=1$ is bounded below by the average
  of $q_2$. Hence, to establish \eqref{linsyshypboundary1}, it is enough to
  prove
  \begin{equation}\label{linsyshypboundary2}
  \begin{aligned}
    \frac1\delta\int_0^s\int_0^\delta q_1\d x\d t&\longrightarrow0,\\
    \frac1\delta\int_0^s\int_{1-\delta}^1q_2\d x\d t&\longrightarrow0.
  \end{aligned}
  \end{equation}
  {
  If $m_1=0$, then $q_1=0$ and the first limit is immediate. Suppose
  $m_1\geq1$ and let $b_*:=\min_{1\leq i\leq m_1}b_i>0$.
  For sufficiently small $\delta$, split the integral at $t=\delta/b_*$. If $t\geq\delta/b_*$ and $0<x<\delta$, the backward characteristic for each $1\leq i\leq m_1$ reaches $x=0$, and hence
  \[
    u_i(t,x)=g_{D,i}(t-x/b_i).
  \]
  Since $\widehat u_i(t,0)=u_i(t,0)=g_{D,i}(t)$, the Lipschitz continuity of $g_{D,i}$ and $\widehat u_i$ gives
  \[
    |u_i(t,x)-\widehat u_i(t,x)|\leq Lx,
    \qquad
    0\leq q_1(t,x)
    \leq
    C\sum_{i=1}^{m_1}|u_i(t,x)-\widehat u_i(t,x)|^2
    \leq Cx^2,
  \]
  where $L$ and $C$ are independent of $\delta$. Using the boundedness of
  $\vec u$ and $\hatu$ on the remaining time interval, we obtain
  \[
  \begin{aligned}
    \frac1\delta
    \int_{\delta/b_*}^s\int_0^\delta q_1\d x\d t
    &\leq
    \frac C\delta
    \int_{\delta/b_*}^s\int_0^\delta x^2\d x\d t
    \longrightarrow0,\\
    \frac1\delta
    \int_0^{\delta/b_*}\int_0^\delta q_1\d x\d t
    &\leq
    \frac C\delta\frac{\delta}{b_*}\delta
    \longrightarrow0.
  \end{aligned}
  \]
  The first limit in \eqref{linsyshypboundary2} follows. The second is
  immediate if $m_2=m$, and otherwise follows from the same calculation at
  $x=1$, with $1-x$ replacing $x$. Hence, the sign inequalities and
  \eqref{linsyshypboundary2} give \eqref{linsyshypboundary1}.

  Using \eqref{linsyshypboundary1} and the cutoff limits established after
  \eqref{slreb_1}, letting $\delta\searrow0$ in \eqref{reb-yl2} gives
  \begin{equation}\label{reb-yl-int-hyp}
  \begin{aligned}
    \frac12\Norm{\vec M^{1/2}\vec e(s)}_{\leb{2}(\W)}^2
    &\leq
    \frac12\Norm{\vec M^{1/2}\vec e(0)}_{\leb{2}(\W)}^2
    -\int_0^s\int_\W\vec r_1\cdot\vec M\vec e\d x\d t\\
    &\leq
    \frac12\Norm{\vec M^{1/2}\vec e(0)}_{\leb{2}(\W)}^2
    +\Norm{\vec M^{1/2}\vec r_1}_{\leb{1}(0,s;\leb{2}(\W))}^2
    +\frac14\Norm{\vec M^{1/2}\vec e}_{\leb{\infty}(0,s;\leb{2}(\W))}^2.
  \end{aligned}
  \end{equation}
  }
  Since the right-hand side of \eqref{reb-yl-int-hyp} is nondecreasing in $s$,
  taking the essential supremum over $s\in(0,t)$ and rearranging gives the
  stated estimate.
  }
\end{proof}

\begin{corollary}[{Fully discrete a posteriori error bound}]
  \label{cor:lsysapost}
  {
  Suppose $\{\vec u_h^i\}_{i=0}^N\subset\mathbb{V}_h^p$ is an RKdG approximation with
  an (IMEX) RK temporal discretization of order at most 3, temporal
  reconstruction $\vut$ (Definition \ref{def:grec}), space-time
  reconstruction $\vuts$ (Definition \ref{def:str}), and residuals
  $\vec r_1$ and $\vec r_2$ defined by \eqref{eq:r1}--\eqref{eq:r2}.
  Assume that $\vec u$ satisfies the assumptions imposed on the entropy
  solution in Theorem \ref{the:yl} when $\veps>0$ and Theorem \ref{cor:yl}
  when $\veps=0$, and that $\vuts$ satisfies the same boundary data as
  $\vec u$.
  Then, for each $i=0,\dots,N$,}
  {
  \begin{multline*}
    \frac 1 4
    \Norm{\vec M^{1/2}\big(\vec u(t_i) - \vec u_h^i\big)}_{\leb{2}(\W)}^2
    +
    \frac{\veps}{2} \enorm{\vec u - \vut}_{\leb{2}(0,t_i;{\mathrm{dG}})}^2
    \\
    \leq
    2 \Big(\Norm{\vec M^{1/2}\big({\vec u_0} - \vuts(0)\big)}_{\leb{2}(\W)}^2
    + 4\Norm{\vec M^{1/2}\vec r_1}_{\leb{1}(0,t_i;\leb{2}( \W))}^2
    + \veps \Norm{\vec M\vec r_2}_{\leb{2}(0,t_i; \sobh{-1}(\W))}^2\Big)\\
    +
    \frac 1 2
    \Norm{\vec M^{1/2}\big(\vuts(t_i) - \vec u_h^i\big)}_{\leb{2}(\W)}^2
    +
    \veps \enorm{\vuts - \vut}_{\leb{2}(0,t_i;{\mathrm{dG}})}^2
  \end{multline*}
  }
\end{corollary}
\begin{proof}
  {
  The stated estimate follows from the error splitting and triangle inequality
  in the proof of Corollary \ref{cor:lsapost} and from Theorem \ref{the:yl} for
  $\veps>0$ and Theorem \ref{cor:yl} for $\veps=0$, with $\hatu=\vuts$.
  }
\end{proof}

%-----------------------------------------------------------------------------------------------------------------------------------------
\section{{Linear wave system with partial diffusion}}
\label{sec:lw}
{We consider the linear system}
\begin{equation}\label{eq:lw}
  \begin{aligned}
    \partial_t u - \partial_x v &= 0,\\
    \partial_t v - \partial_x u &= \veps \del_{xx} v,
  \end{aligned}
\end{equation}
{in which diffusion acts only on the second state variable $v$.
With $\vec u=(u,v)^T$, this system has the form \eqref{eq:cd} with
$\vec f(\vec u)=\vec B\vec u$ and $\vec A(\vec u)=\vec A$, where}
\begin{equation*}
  \vec B =
  \begin{pmatrix}
    0 & -1 \\ -1 & 0
  \end{pmatrix} \quad \text{and} \quad
  \vec A =
  \begin{pmatrix}
    0 & 0 \\ 0 & 1
  \end{pmatrix}.
\end{equation*}
{Eliminating $v$ from \eqref{eq:lw} gives the damped wave equation}
\begin{equation*}
  \del_{tt} u - \del_{xx} u = \veps \del_{txx} u.
\end{equation*}

{We focus on the viscous case $\veps>0$, since for $\veps=0$,
system \eqref{eq:lw} reduces to a linear hyperbolic system of the form
considered in Section \ref{sec:yl}, and no separate analysis is needed. We
impose periodic boundary conditions on $\mathbb{T}$ throughout this section
because the boundary data for the diffusive variable $v$ and the incoming
characteristic fields must be compatible.}

{Let $\widehat{\vec u}=(\widehat u,\widehat v)^T$ satisfy the perturbed system}
\begin{equation}\label{eq:lwr}
  \begin{aligned}
    \partial_t \widehat u - \partial_x \widehat v
    &=r_u,\\
    \partial_t \widehat v - \partial_x \widehat u
    &=\veps \del_{xx} \widehat v + r_{v,1} + \veps r_{v,2},
  \end{aligned}
\end{equation}
{where $r_u,r_{v,1}\in\leb{2}((0,T)\times\mathbb{T})$ and
$r_{v,2}\in\leb{2}(0,T;\sobh{-1}(\mathbb{T}))$. We write
$\vec r_1:=(r_u,r_{v,1})^T$.}

\begin{remark}[Entropy]\label{rem:ent-lw}
  {Any strictly convex function
  $\eta\in \cont{2}(\reals^2, \reals)$ satisfying
  $\eta_{uu}=\eta_{vv}$ gives rise to an entropy/entropy-flux pair for
  \eqref{eq:lw}. Among these choices, we use the quadratic entropy-flux pair}
  \begin{equation}\label{eq:lw-entropy}
    \eta(u,v)= \frac{1}{2} u^2 + \frac{1}{2} v^2, \quad q(u,v) = -uv,
  \end{equation}
  {since its relative entropy is
  $\frac{1}{2}|\vec u-\widehat{\vec u}|^2$ and hence gives the $L^2$-error directly.
  Note that $\eta$ can be interpreted as the mechanical energy density, and $q$
  is the associated energy flux for the inviscid part of the system.}
\end{remark}

\begin{theorem}[{Stability estimate for the viscous problem}]\label{the:lw}
  {For $\veps>0$, let $\vec u=(u,v)^T$ be an entropy solution
  of \eqref{eq:lw} with respect to the entropy \eqref{eq:lw-entropy} and with
  initial data {$\vec u_0\in\leb{2}(\mathbb{T};\reals^2)$}. We assume that
  $u\in\leb{2}((0,T)\times\mathbb{T})$ and
  $v\in\leb{2}(0,T;\sobh{1}(\mathbb{T}))$.
  Let $\widehat{\vec u}\in
  \sob{1}{\infty}((0,T)\times\mathbb{T};\reals^2)$ be a solution of
  \eqref{eq:lwr} with initial data
  $\widehat{\vec u}_0\in\sob{1}{\infty}(\mathbb{T};\reals^2)$.
  Assume that both $\vec u$ and $\widehat{\vec u}$ satisfy periodic boundary
  conditions. Then, for almost every $t\in(0,T)$,}
  {
  \begin{multline*}
    \frac 1 4\Norm{\vec u - \widehat{\vec u}}_{\leb{\infty}(0,t;\leb{2}(\mathbb{T}))}^2 + \frac \veps 2
    \Norm{\partial_x (v - \widehat v)}_{\leb{2}(0,t;\leb{2}(\mathbb{T}))}^2
    \leq
    \Norm{\vec u_0 - \widehat{\vec u}_0}_{\leb{2}(\mathbb{T})}^2
    + 4\Norm{\vec r_1}_{\leb{1}(0,t;\leb{2}(\mathbb{T};\reals^2))}^2 \\
    + \veps
    \Norm{r_{v,2}}_{\leb{2}(0,t;\sobh{-1}(\mathbb{T}))}^2.
  \end{multline*}
  }
\end{theorem}

\begin{proof}
  {
  Set $\vec e:=\vec u-\widehat{\vec u}=(e_u,e_v)^T$.
  For the entropy/entropy-flux pair \eqref{eq:lw-entropy}, the relative entropy and relative entropy flux are
  \begin{equation*}
    \eta(\vec u|\widehat{\vec u})=\frac12|\vec e|^2,
    \qquad
    q(\vec u|\widehat{\vec u})=-e_ue_v.
  \end{equation*}
  Since $\D^2\eta=I$ and $\vec A=\operatorname{diag}(0,1)$, the diffusion
  terms in Lemma \ref{lem:greb} reduce to
  \[
    -\veps\partial_xe_v\,\partial_xv
    +\veps\partial_x\widehat v\,\partial_xe_v
    =-\veps|\partial_xe_v|^2.
  \]
  {Fix $0<t<T$ and choose $0<\delta<T-t$. Let $\Phi_\delta(x,s)=\zeta_\delta(s)$,
  where $\zeta_\delta$ is the temporal cutoff from \eqref{def:psidelta} with
  cutoff time $t$. Since $\partial_x\Phi_\delta=0$, the terms involving
  $\partial_x\Phi_\delta$ vanish upon applying Lemma \ref{lem:greb} with this
  test function, and no spatial boundary term arises on $\mathbb{T}$. Since
  $\Phi_\delta e_v\to\mathbf 1_{(0,t)}e_v$ in
  $\leb{2}(0,T;\sobh{1}(\mathbb{T}))$, we have}
  \[
    \langle r_{v,2},\Phi_\delta e_v\rangle
    \longrightarrow
    \int_0^t\langle r_{v,2}(s),e_v(s)\rangle\d s.
  \]
  {H\"older's and Young's inequalities give}
  \begin{align*}
    \veps
    \left|
    \int_0^t\langle r_{v,2}(s),e_v(s)\rangle\d s
    \right|
    &\leq
    \veps
    \Norm{r_{v,2}}_{\leb{2}(0,t;\sobh{-1}(\mathbb{T}))}
    \Norm{\partial_xe_v}_{\leb{2}((0,t)\times\mathbb{T})}\\
    &\leq
    \frac{\veps}{2}
    \Norm{r_{v,2}}_{\leb{2}(0,t;\sobh{-1}(\mathbb{T}))}^2
    +
    \frac{\veps}{2}
    \Norm{\partial_xe_v}_{\leb{2}((0,t)\times\mathbb{T})}^2.
  \end{align*}
  Passing to the limit $\delta\searrow0$ in the remaining terms and applying H\"older's and Young's inequalities, we obtain, for almost every $t\in(0,T)$,
  \begin{equation}\label{eq:lw-full-estimate}
  \begin{aligned}
    \frac12\Norm{\vec e(t)}_{\leb{2}(\mathbb{T})}^2
    +\frac{\veps}{2}
    \Norm{\partial_x e_v}_{\leb{2}(0,t;\leb{2}(\mathbb{T}))}^2
    \leq{}&
    \frac12\Norm{\vec e(0)}_{\leb{2}(\mathbb{T})}^2
    +2\Norm{\vec r_1}_{\leb{1}(0,t;\leb{2}(\mathbb{T};\reals^2))}^2\\
    &+\frac18\Norm{\vec e}_{\leb{\infty}(0,t;\leb{2}(\mathbb{T}))}^2
    +\frac{\veps}{2}
    \Norm{r_{v,2}}_{\leb{2}(0,t;\sobh{-1}(\mathbb{T}))}^2.
  \end{aligned}
  \end{equation}
  {Since both terms on the left-hand side of \eqref{eq:lw-full-estimate} are nonnegative, each is bounded separately by the right-hand side. Taking the essential supremum over $s\in(0,t)$ in the estimate for $\Norm{\vec e(s)}_{\leb{2}(\mathbb{T})}^2$ and using that the norms on the right-hand side are nondecreasing in $s$ gives an estimate for $\Norm{\vec e}_{\leb{\infty}(0,t;\leb{2}(\mathbb{T}))}^2$. Adding this estimate to the estimate for $\Norm{\partial_xe_v}_{\leb{2}((0,t)\times\mathbb{T})}^2$ and rearranging gives the stated result.}
  }
\end{proof}

\begin{corollary}[{Fully discrete a posteriori error bound}]
  \label{cor:lwaveapost}
  {Suppose $\{\vec u_h^i\}_{i=0}^N \subset \mathbb{V}_h^p$
  is an RKdG approximation
  with an (IMEX) RK temporal discretization of order at most 3,
  temporal reconstruction $\widehat{\vec u}_h^{\mathrm t}$ (Definition \ref{def:grec}),
  space-time reconstruction $\widehat{\vec u}^{\mathrm{ts}}$ (Definition \ref{def:str}),
  and residuals $\vec r_1$ and $r_{v,2}$ defined by \eqref{eq:lwr}.
  Assume that $\vec u$ satisfies the assumptions imposed on the entropy
  solution in Theorem \ref{the:lw} when $\veps>0$ and in Theorem
  \ref{cor:yl}, with periodic boundary conditions in place of inflow boundary
  conditions, when $\veps=0$.
  Then, for each $i = 0,\dots, N$,}
  {
  \begin{multline*}
    \frac 1 4\Norm{\vec u(t_i) - \vec u_h^i}_{\leb{2}(\mathbb{T})}^2
    + \frac \veps 2 \enorm{v - \widehat v_h^{\mathrm t}}_{\leb{2}(0,t_i;\mathrm{dG})}^2 \\
    \leq
    2\bigg(
    \Norm{\vec u_0 - \widehat{\vec u}^{\mathrm{ts}}(0)}_{\leb{2}(\mathbb{T})}^2
    +4\Norm{\vec r_1}_{\leb{1}(0,t_i;\leb{2}(\mathbb{T};\reals^2))}^2\\
    +\veps
    \Norm{r_{v,2}}_{\leb{2}(0,t_i;\sobh{-1}(\mathbb{T}))}^2
    \bigg)\\
    +\frac12\Norm{\widehat{\vec u}^{\mathrm{ts}}(t_i) - \vec u_h^i}_{\leb{2}(\mathbb{T})}^2
    +\veps \enorm{\widehat v^{\mathrm{ts}} - \widehat v_h^{\mathrm t}}_{\leb{2}(0,t_i;\mathrm{dG})}^2
  \end{multline*}
  }
\end{corollary}

\begin{proof}
  {
  Using the error splitting and triangle inequality from the proof of
  Corollary \ref{cor:lsapost}, we apply Theorem \ref{the:lw} for $\veps>0$,
  with $\widehat{\vec u}=\widehat{\vec u}^{\mathrm{ts}}$. For $\veps=0$,
  the argument used to prove Theorem \ref{cor:yl} applies with the same
  choice, while periodicity eliminates the boundary terms. This gives the stated estimate.
  }
\end{proof}

%---------------------------------------------------------------------------------------------------------------------------------------

\section{{Nonlinear wave system with partial diffusion}}\label{sec:nw}
{For an open interval $I\subset\reals$ and a strictly convex function
$W\in\cont{3}(I,\reals)$, we consider $\vec u=(u,v)^T\in U:=I\times\reals$
satisfying the nonlinear wave system}
\begin{equation}\label{eq:nw}
  \begin{split}
    &\partial_t u - \partial_x v = 0,\\
    &\partial_t v - \partial_x W'(u) = \veps \del_{xx} v.
  \end{split}
\end{equation}
The system corresponds to \eqref{eq:cd} with
\begin{equation*}
  \vec f(u,v) = \begin{pmatrix} -v \\ -W'(u) \end{pmatrix}, \quad
  \vec A(u,v) = \begin{pmatrix} 0&0\\ 0& 1 \end{pmatrix}.
\end{equation*}
{The diffusion matrix $\vec A$ is singular, so only $v$ is
diffused. Eliminating $v$ from \eqref{eq:nw} gives}
\begin{equation*}
  \del_{tt} u - \del_{xx} W'(u) = \veps \del_{txx} u.
\end{equation*}

{The boundary data for the diffusive variable $v$ and the
incoming characteristic fields must also be compatible for \eqref{eq:nw},
so we impose periodic boundary conditions on $\mathbb{T}$ throughout this
section.}

{Let $\widehat{\vec u}=(\widehat u,\widehat v)^T$ be an
approximate strong solution satisfying the perturbed system}
  \begin{equation}\label{eq:nwr}
    \begin{split}
      &\partial_t \widehat u - \partial_x \widehat v = r_u,\\
      &\partial_t \widehat v - \partial_x W'(\widehat u) = \veps \del_{xx} \widehat v + r_{v,1} + \veps r_{v,2},
    \end{split}
  \end{equation}
{where the residuals satisfy}
\begin{equation*}
  r_u,r_{v,1}\in\leb{2}((0,T)\times\mathbb{T}),
  \qquad
  r_{v,2}\in\leb{2}(0,T;\sobh{-1}(\mathbb{T})).
\end{equation*}

\begin{remark}[Entropy]\label{rem:ent-nw}
  {Any strictly convex function
  $\eta\in\cont{2}(U,\reals)$ satisfying
  $\eta_{uu}=W''(u)\eta_{vv}$ gives rise to an entropy/entropy-flux pair for
  \eqref{eq:nw}. Among these choices, we use the entropy/entropy-flux pair}
  \begin{equation}\label{eq:nw-entropy}
    \eta(u,v)=W(u)+\frac{1}{2}v^2,\quad q(u,v)=-vW'(u).
  \end{equation}
  {For $W(u)=\tfrac{1}{2}u^2$, this pair reduces to the quadratic
  entropy/entropy-flux pair \eqref{eq:lw-entropy}. We write
  $W(a|b):=W(a)-W(b)-W'(b)(a-b)$ for $a,b\in I$.}
  % This represents the mechanical energy and its flux of the system.
\end{remark}

\begin{assumption}[A priori control on the nonlinearity]\label{ass:nw-apriori}
  {We assume that, for some $T>0$, there is a compact interval
  $\mathfrak{K}\subset I$ such that $u(t,x),\widehat u(t,x)\in\mathfrak{K}$
  for almost every $(t,x)\in(0,T)\times\mathbb{T}$ and $W''(z)>0$ for every
  $z\in\mathfrak K$. We define}
  \begin{equation}\label{eq:bi}
    {c_W:=\min_{z\in\mathfrak K}W''(z)>0,\qquad
    C_W:=\max_{z\in\mathfrak K}|W'''(z)|<\infty.}
  \end{equation}
\end{assumption}

\begin{theorem}[{Stability estimate for the viscous problem}]\label{the:nw}
  {Under Assumption~\ref{ass:nw-apriori}, let $\veps>0$.
  Let $\vec u=(u,v)^T$ be an entropy solution of \eqref{eq:nw} with respect
  to the entropy \eqref{eq:nw-entropy} and with initial data
  {$\vec u_0=(u_0,v_0)^T \in \leb{2}(\mathbb{T}; \reals^2)$}. Assume that
  $u \in \leb{\infty}((0,T) \times \mathbb{T})$ and
  $v \in \leb{2}(0,T;\sobh{1}(\mathbb{T}))$.
  Let $\widehat{\vec u}
  \in \sob{1}{\infty}((0,T)\times \mathbb{T}; \reals^2)$
  be a solution of \eqref{eq:nwr} with initial data
  $\widehat{\vec u}_0=(\widehat u_0,\widehat v_0)^T
  \in \sob{1}{\infty}(\mathbb{T};\reals^2)$.
  Assume that both $\vec u$ and $\widehat{\vec u}$ satisfy periodic boundary
  conditions. Then, for almost every $t \in (0,T)$,}
  \begin{multline*}
    \frac{c_W}{4}\Norm{u - \widehat u}_{\leb{\infty}(0,t;\leb{2}(\mathbb{T}))}^2
    + \frac{1}{4}\Norm{v - \widehat v}_{\leb{\infty}(0,t;\leb{2}(\mathbb{T}))}^2
    + \frac{\veps}{2} \Norm{\partial_x(v - \widehat v)}_{\leb{2}((0,t)\times\mathbb{T})}^2 \\
    \leq {2}\bigg[2\int_{\mathbb{T}} {W(u_0|\widehat u_0)} \d x
    + {\Norm{v_0 - \widehat v_0}_{\leb{2}(\mathbb{T})}^2}
    + \frac{4}{c_W} \Norm{W''(\widehat u) r_u}_{\leb{1}(0,t;\leb{2}(\mathbb{T}))}^2 \\
    + 4 \Norm{r_{v,1}}_{\leb{1}(0,t;\leb{2}(\mathbb{T}))}^2
    + \veps
    \Norm{r_{v,2}}_{\leb{2}(0,t;\sobh{-1}(\mathbb{T}))}^2\bigg]
    {\exp\bigg(\frac{2C_W}{c_W}t \Norm{\partial_x \widehat v}_{\leb{\infty}((0,t)\times \mathbb{T})}\bigg)}
  \end{multline*}
\end{theorem}

\begin{proof}
  {
  For the entropy in \eqref{eq:nw-entropy}, we have
  \begin{equation*}
    \eta(\vec u|\hatu)
    =W(u|\widehat u)+\tfrac{1}{2}(v-\widehat v)^2,
    \qquad
    \D^2 \eta (\vec u) = \begin{pmatrix}
      W''(u) & 0\\ 0 & 1
    \end{pmatrix}.
  \end{equation*}
  }
  {We have
  \[
    \D\eta(\vec u)-\D\eta(\hatu)
    -\Transpose{\big(\D^2\eta(\hatu)(\vec u-\hatu)\big)}
    =(W'(u|\widehat u),0).
  \]
  The relative flux is $\vec f(\vec u|\hatu) = (0, -W'(u|\widehat u))^T$, where
  $W'(u|\widehat u) := W'(u) - W'(\widehat u) - W''(\widehat u)(u - \widehat u)$.}
  % Since $\D \eta(\vec u) - \D \eta (\hatu) = (W'(u) - W'(\widehat u), v - \widehat v)^T$, we have
  % $\D\eta(\vec u| \hatu) = (W'(u|\widehat u), 0)^T$ and 
  % $\vec f(\vec u|\hatu) = (0, -W'(u|\widehat u))^T$
  % with $W'(u|\widehat u) = W'(u) - W'(\widehat u) - W''(\widehat u)(u - \widehat u)$.
  Let $\vec e := \vec u - \hatu$ with components $e_u := u - \widehat u$ and $e_v := v - \widehat v$.
  {
  Since $\vec A=\begin{pmatrix}0&0\\0&1\end{pmatrix}$, we obtain
  \begin{align*}
    &-\veps\partial_x\big(\D\eta(\vec u)-\D\eta(\hatu)\big)
    \vec A\partial_x\vec u
    +\veps\big(\vec A\partial_x\hatu\big)\cdot
    \partial_x\big(\D^2\eta(\hatu)\vec e\big)\\
    &\qquad
    =-\veps\partial_xe_v\,\partial_xv
    +\veps\partial_x\widehat v\,\partial_xe_v
    =-\veps|\partial_xe_v|^2.
  \end{align*}
  {Applying Lemma \ref{lem:greb} with $\Phi(t,x)=\zeta_\delta(t)$, where $\zeta_\delta$ is defined in \eqref{def:psidelta}, the terms involving $\partial_x\Phi$ vanish and the spatial boundary terms cancel by periodicity. Letting $\delta\searrow0$ gives, for almost every $t\in(0,T)$,}
  }
  \begin{multline}
    \label{nw-reb-int}
      \int_{\mathbb{T}} W(u(t)|\widehat u(t)) \d x
      +
      \frac 1 2 \Norm{e_v(t)}_{\leb{2}(\mathbb{T})}^2
      +
      \veps \Norm{\partial_x e_v}_{\leb{2}((0,t)\times \mathbb{T})}^2
      \\
      \leq
      \int_{\mathbb{T}} {W(u_0|\widehat u_0)} \d x
      +
      \frac 1 2 {\Norm{v_0-\widehat v_0}_{\leb{2}(\mathbb{T})}^2}
      - \veps{\int_0^t\langle r_{v,2}(s),e_v(s)\rangle\d s} \\
      +
      \int_0^t \int_{\mathbb{T}}  (\partial_x \widehat v) W'(u | \widehat u)
      -
      r_u W''(\widehat u) e_u - r_{v,1} e_v\d x \d s .
  \end{multline}
  {The term involving $r_{v,2}$ is estimated as in the proof of
  Theorem \ref{the:lw}.}
  From \eqref{eq:bi}, we have $2 W(u| \widehat u) \geq c_W \norm{ u - \widehat u}^2$ and $2 |W'(u| \widehat u)| \leq C_W \norm{ u - \widehat u}^2$ for all $u, \widehat u \in \mathfrak{K}$.
  {Substituting these estimates into \eqref{nw-reb-int} and
  applying Young's inequality to the remaining residual terms yields}
  \begin{multline} \label{nw-reb-int2}
      \frac{c_W}{2} \Norm{e_u(t) }_{\leb{2}(\mathbb{T})}^2
      +
      \frac{1}{2}\Norm{e_v(t)}_{\leb{2}(\mathbb{T})}^2
      +
      \frac{\veps}{2} \Norm{\partial_x e_v}_{\leb{2}((0,t)\times\mathbb{T})}^2
      \\ \leq
      \int_{\mathbb{T}} {W(u_0|\widehat u_0)} \d x
      +
      \frac 1 2 {\Norm{v_0-\widehat v_0}_{\leb{2}(\mathbb{T})}^2}
      +
      \frac{C_W}{2}\Norm{\partial_x \widehat v}_{\leb{\infty}((0,t)\times \mathbb{T})}
      \Norm{ e_u}_{\leb{2}((0,t)\times \mathbb{T})}^2 \\
      \quad + \frac{2}{c_W} \Norm{W''(\widehat u) r_u}_{\leb{1}(0,t;\leb{2}(\mathbb{T}))}^2 + \frac{c_W}{8}\Norm{e_u}_{\leb{\infty}(0,t;\leb{2}(\mathbb{T}))}^2\\
      \quad + 2 \Norm{r_{v,1}}_{\leb{1}(0,t;\leb{2}(\mathbb{T}))}^2 + \frac{1}{8}\Norm{e_v}_{\leb{\infty}(0,t;\leb{2}(\mathbb{T}))}^2
    + \frac{\veps}{2}
    \Norm{r_{v,2}}_{\leb{2}(0,t;\sobh{-1}(\mathbb{T}))}^2.
  \end{multline}
  {
  Fix $0<t<T$ and apply \eqref{nw-reb-int2} with $s\in(0,t)$ in place of $t$.
  The two terms on the right-hand side of \eqref{nw-reb-int2} containing the
  $L^\infty(0,t;L^2(\mathbb{T}))$ norms of $e_u$ and $e_v$ satisfy
  \begin{multline*}
    \frac{c_W}{8}
    \Norm{e_u}_{\leb{\infty}(0,t;\leb{2}(\mathbb{T}))}^2
    +\frac18
    \Norm{e_v}_{\leb{\infty}(0,t;\leb{2}(\mathbb{T}))}^2\\
    \leq
    \frac12\operatorname*{ess\,sup}_{0<s<t}
    \bigg[
    \frac{c_W}{2}\Norm{e_u(s)}_{\leb{2}(\mathbb{T})}^2
    +\frac12\Norm{e_v(s)}_{\leb{2}(\mathbb{T})}^2
    +\frac{\veps}{2}
    \Norm{\partial_xe_v}_{\leb{2}((0,s)\times\mathbb{T})}^2
    \bigg].
  \end{multline*}
  Taking the essential supremum over $s\in(0,t)$ in \eqref{nw-reb-int2},
  moving this term to the left, and applying Gronwall's inequality therefore give
  \begin{multline*}
    \operatorname*{ess\,sup}_{0<s<t}
    \bigg[
    \frac{c_W}{2}\Norm{e_u(s)}_{\leb{2}(\mathbb{T})}^2
    +\frac12\Norm{e_v(s)}_{\leb{2}(\mathbb{T})}^2
    +\frac{\veps}{2}
    \Norm{\partial_xe_v}_{\leb{2}((0,s)\times\mathbb{T})}^2
    \bigg]\\
    \leq
    \bigg[
    2\int_{\mathbb{T}}W(u_0|\widehat u_0)\d x
    +\Norm{v_0-\widehat v_0}_{\leb{2}(\mathbb{T})}^2
    +\frac4{c_W}
    \Norm{W''(\widehat u)r_u}_{\leb{1}(0,t;\leb{2}(\mathbb{T}))}^2\\
    +4\Norm{r_{v,1}}_{\leb{1}(0,t;\leb{2}(\mathbb{T}))}^2
    +\veps
    \Norm{r_{v,2}}_{\leb{2}(0,t;\sobh{-1}(\mathbb{T}))}^2
    \bigg]
    \exp\bigg(
    \frac{2C_W}{c_W}t
    \Norm{\partial_x\widehat v}_{\leb{\infty}((0,t)\times\mathbb{T})}
    \bigg).
  \end{multline*}
  The sum of the first two terms on the left-hand side of the theorem is
  bounded by the essential supremum in the preceding estimate, and the third
  term is bounded by the same quantity. Hence, the entire left-hand side is
  bounded by twice this essential supremum, and the stated estimate follows.
  }
\end{proof}

\begin{theorem}[{Stability estimate for the inviscid problem}]
  \label{cor:nw}
  {Under Assumption~\ref{ass:nw-apriori}, let $\veps=0$.
  Let $\vec u
  \in \leb{\infty}((0,T)\times\mathbb{T}; \reals^2)$
  be an entropy solution of \eqref{eq:nw} with respect to the entropy
  \eqref{eq:nw-entropy} and with initial data
  $\vec u_0=(u_0,v_0)^T
  \in \leb{\infty}(\mathbb{T}; \reals^2)$.
  Let $\widehat{\vec u}
  \in \sob{1}{\infty}((0,T)\times\mathbb{T}; \reals^2)$
  be a solution of \eqref{eq:nwr} with initial data
  $\widehat{\vec u}_0=(\widehat u_0,\widehat v_0)^T
  \in \sob{1}{\infty}(\mathbb{T}; \reals^2)$.
  Assume that both $\vec u$ and $\widehat{\vec u}$ satisfy periodic boundary
  conditions. Then, for almost every $t \in (0,T)$,}
  \begin{multline*}
    \frac{c_W}{4}\Norm{u(t) - \widehat u(t)}_{\leb{2}(\mathbb{T})}^2 + \frac{1}{4}\Norm{v(t) - \widehat v(t)}_{\leb{2}(\mathbb{T})}^2 \\
    \leq \bigg[\int_{\mathbb{T}}
    {W(u_0|\widehat u_0)} \d x
    + \frac{1}{2}
    {\Norm{v_0 - \widehat v_0}_{\leb{2}(\mathbb{T})}^2} \\
    + \frac{2}{c_W} \Norm{W''(\widehat u) r_u}_{\leb{1}(0,t;\leb{2}(\mathbb{T}))}^2
    + 2 \Norm{r_{v,1}}_{\leb{1}(0,t;\leb{2}(\mathbb{T}))}^2\bigg]
    {\exp\bigg(\frac{2C_W}{c_W}t \Norm{\partial_x \widehat v}_{\leb{\infty}((0,t)\times \mathbb{T})}\bigg)}
  \end{multline*}
\end{theorem}
\begin{proof}
  {
  Repeating the argument leading to \eqref{nw-reb-int2} with $\veps=0$
  gives that inequality with the diffusion terms omitted, and the two terms on
  its right-hand side containing the $L^\infty(0,t;L^2(\mathbb{T}))$ norms of
  $e_u$ and $e_v$ are estimated as in the viscous case. Taking the essential
  supremum over $s\in(0,t)$, moving these terms to the left, and applying
  Gronwall's inequality therefore give the stated estimate.
  }
\end{proof}

\begin{corollary}[{Fully discrete a posteriori error bound}]
  \label{cor:nlsysapost}
  {
  Suppose $\{(u_h^i, v_h^i)\}_{i=0}^N \subset \mathbb{V}_h^p$ is an RKdG approximation
  using an (IMEX) RK temporal discretization of order at most 3, with temporal
  reconstruction $(\widehat u_h^{\mathrm t}, \widehat v_h^{\mathrm t})$
  (Definition \ref{def:grec}), space-time reconstruction
  $(\widehat u^{\mathrm{ts}}, \widehat v^{\mathrm{ts}})$ (Definition \ref{def:str}), and
  residuals $r_u, r_{v,1}, r_{v,2}$ defined by \eqref{eq:nwr}.
  Assume that $\vec u$ satisfies the assumptions imposed on the entropy
  solution in Theorem \ref{the:nw} when $\veps>0$ and Theorem \ref{cor:nw}
  when $\veps=0$.
  Then, for each $i = 0,\dots, N$, we have
  }
  \begin{multline*}
    \frac{c_W}{4}\Norm{u(t_i) - u_h^i}_{\leb{2}({\mathbb{T}})}^2
    + \frac{1}{4}\Norm{v(t_i) - v_h^i}_{\leb{2}({\mathbb{T}})}^2
    + \frac{\veps}{2}
    \enorm{v - \widehat v_h^{\mathrm t}}_{\leb{2}(0,t_i;{\mathrm{dG}})}^2\\
    \leq
    {8\exp\bigg(\frac{2C_W}{c_W}t_i \Norm{\partial_x \widehat v^{\mathrm{ts}}}_{\leb{\infty}((0,t_i) \times \mathbb{T})}\bigg)}
    \bigg[\int_{{\mathbb{T}}} {W(u_0|\widehat u^{\mathrm{ts}}(0))} \d x
    + \frac{1}{2}\Norm{{v_0} - \widehat v^{\mathrm{ts}}(0)}_{\leb{2}({\mathbb{T}})}^2 \\
    + \frac{2}{c_W}\Norm{W''(\widehat u^{\mathrm{ts}}) r_u}_{\leb{1}(0,t_i;\leb{2}({\mathbb{T}}))}^2
    + 2\Norm{r_{v,1}}_{\leb{1}(0,t_i;\leb{2}({\mathbb{T}}))}^2
    + \frac{\veps}{2}
    \Norm{r_{v,2}}_{\leb{2}(0,t_i; \sobh{-1}({\mathbb{T}}))}^2\bigg] \\
    + \frac{c_W}{2}\Norm{\widehat u^{\mathrm{ts}}(t_i) - u_h^i}_{\leb{2}({\mathbb{T}})}^2
    + \frac{1}{2}\Norm{\widehat v^{\mathrm{ts}}(t_i) - v_h^i}_{\leb{2}({\mathbb{T}})}^2
    + \veps \enorm{\widehat v^{\mathrm{ts}} - \widehat v_h^{\mathrm t}}_{\leb{2}(0,t_i;{\mathrm{dG}})}^2.
  \end{multline*}
\end{corollary}

\begin{proof}
  {
  Using the error splitting and triangle inequality from the proof of
  Corollary \ref{cor:lsapost}, we apply Theorem \ref{the:nw} for $\veps>0$
  and Theorem \ref{cor:nw} for $\veps=0$, in both cases with
  $(\widehat u,\widehat v)=(\widehat u^{\mathrm{ts}},\widehat v^{\mathrm{ts}})$, to obtain the
  stated estimate.
  }
\end{proof}

\begin{remark}[Computability of the full a posteriori error bound]
  {Evaluation of the full a posteriori error bound in
  Corollary~\ref{cor:nlsysapost} requires the constants $c_W$ and $C_W$
  defined in Assumption~\ref{ass:nw-apriori}. These constants are determined
  on the interval $\mathfrak K$, which contains the values of both $u$ and
  $\widehat u^{\mathrm{ts}}$. The range of $\widehat u^{\mathrm{ts}}$ is
  computable, but an enclosure of the range of $u$ must be supplied a priori.}
\end{remark}

\section{Numerical experiments}
\label{sec:num}

{
To examine the scaling of the residual quantities, we compare their
convergence rates with those of the discretization errors in the
convection-dominated regime. {We assess their robustness with
respect to viscosity using effectivity indices computed as $\veps$
decreases.}
}

{
{The numerical experiments are implemented in \texttt{Julia}
\cite{2017Bezanson}, and the code is available in the
\href{https://github.com/kwkwon13/aposteriori-cd-experiments}{project
repository}\footnote{\url{https://github.com/kwkwon13/aposteriori-cd-experiments}}.}
For spatial discretization, we use the discontinuous Galerkin scheme
\eqref{eq:sdisc}, with the symmetric interior penalty (SIP) method for
diffusion \cite{Di-PietroErn:2012} and the Lax-Wendroff flux \eqref{def:LW}
for convection, and for time integration we use \texttt{KenCarp3} from
\texttt{OrdinaryDiffEq.jl}, a third-order implicit-explicit (IMEX)
Runge-Kutta method \cite{Kennedy2003}.
}

{
The hyperbolic residual quantity
$\Norm{\vec r_1}_{\leb{1}(0,T;\leb{2}(\W))}$ is evaluated directly from its
definition in \eqref{eq:r1}. When a manufactured solution is used, $\vec r_1$
denotes the residual of the balance law with the prescribed source term.
Taking the $\leb{2}(0,T)$-norm of the pointwise estimate in Theorem
\ref{thm:h-1_est} gives
\begin{equation*}
  \Norm{\vec r_2}_{\leb{2}(0,T;\sobh{-1}(\W))}
  \leq
  C\Norm{\mathcal E_2}_{\leb{2}(0,T)},
\end{equation*}
where $\mathcal E_2(t)$ is the computable estimator defined in that theorem.
}

{
To compare the convergence rates of the discretization errors and residual
quantities, we compute the experimental order of convergence (EOC) as
\begin{equation*}
\text{EOC} = \frac{\log(e_{h_1}/e_{h_2})}{\log(h_1/h_2)},
\end{equation*}
where $e_{h_i}$ denotes a discretization error or residual quantity evaluated
on a mesh of size $h_i$, and $h_1 > h_2$.
}

{
We compare $\vec r_1$ and the computable estimator $\mathcal E_2$ separately
with the discretization errors using the effectivity indices
\begin{equation*}
\mathrm{EI}_1 =
\frac{\Norm{\vec r_1}_{\leb{1}(0,T;\leb{2}(\W))}}
{\displaystyle\max_n\Norm{\vec u(t_n)-\vec u_h^n}_{\leb{2}(\W)}},
\qquad
\mathrm{EI}_2 =
\frac{\Norm{\mathcal E_2}_{\leb{2}(0,T)}}
{\enorm{\vec u-\widehat{\vec u}_h^{\mathrm t}}_{\leb{2}(0,T;{\mathrm{dG}})}}.
\end{equation*}
For the nonlinear wave system, both $\mathcal E_2$ and the denominator of
$\mathrm{EI}_2$ are evaluated using only the $v$-component. The indices are
evaluated at each mesh level used in the refinement study.
}

{
We use successively refined uniform spatial meshes with
$N_e\in\{2^4,2^5,\ldots,2^{10}\}$ elements, polynomial degrees
$p\in\{1,2\}$, final time $T=0.5$, and time steps of the form $\Delta t=Ch$,
with $C$ specified for each problem. We display results for the representative viscosities
$\veps\in\{10^{-1},10^{-3},10^{-6},10^{-8}\}$ to examine the transition
toward the convection-dominated regime as $\veps$ decreases.
}

\subsection{Linear advection-diffusion equation}

{
We evaluate the residual quantities in Corollary \ref{cor:lsapost} for
\begin{equation}\label{eq:linear-adv-diff}
  \partial_t u + \partial_x u = \veps \partial_{xx} u
\end{equation}
on $\W=[0,2\pi]$, subject to periodic boundary conditions and the initial
condition $u(0,x)=\sin{x}$. The exact solution is
\begin{equation*}
  u(t,x)=e^{-\veps t}\sin{x-t},
\end{equation*}
and we use $\Delta t=0.1h$.
}

{
The EOCs of the discretization errors and residual quantities are compared in
the first two rows of Figure \ref{fig:linear_refinement_summary}. On the
finest meshes, the maximum error at the time-grid points and the error measured
by the $\leb{2}$-in-time dG energy seminorm have EOCs close to $p+1$ and $p$,
respectively, for all viscosities shown.
}

{
For $\veps\leq10^{-6}$, the EOCs of
$\Norm{r_1}_{\leb{1}(0,T;\leb{2}(\W))}$ and
$\Norm{\mathcal E_2}_{\leb{2}(0,T)}$ on the finest meshes are close to $p+1$
and $p$, respectively, matching those of the two discretization errors. {For
$N_e=1024$, $\mathrm{EI}_1$ and $\mathrm{EI}_2$ lie in the ranges
$1.45$--$1.49$ and $1.53$--$1.86$, respectively. At this mesh level, both
indices change only slightly for each polynomial degree as $\veps$ decreases
from $10^{-6}$ to $10^{-8}$, providing numerical evidence of robustness with
respect to viscosity in the convection-dominated regime.} By contrast, for
$\veps=10^{-1}$ and $N_e=1024$, the EOCs of
$\Norm{r_1}_{\leb{1}(0,T;\leb{2}(\W))}$ are $0.10$ for $p=1$ and $0.98$ for
$p=2$, and those of $\Norm{\mathcal E_2}_{\leb{2}(0,T)}$ are $1.00$ for $p=1$
and $1.55$ for $p=2$. {The lower EOCs of
$\Norm{r_1}_{\leb{1}(0,T;\leb{2}(\W))}$, relative to those of the maximum error
at the time-grid points, explain the growth of $\mathrm{EI}_1$ under mesh
refinement. We therefore omit the $\veps=10^{-1}$ curves from the
$\mathrm{EI}_1$ panels in Figure \ref{fig:linear_refinement_summary}.}
}

\subsection{Viscous Burgers equation}

{
We evaluate the residual quantities in Corollary \ref{cor:nlsapost} for
\begin{equation}\label{eq:viscous-burgers}
  \partial_t u + \partial_x\left(\frac{u^2}{2}\right)
  = \veps \partial_{xx} u + f(t,x)
\end{equation}
on $\W=[0,2\pi]$, subject to periodic boundary conditions. We use the
manufactured solution
\begin{equation*}
  u(t,x)=2+0.5\sin{2x-t}+0.3\cos{x+2t}.
\end{equation*}
The source term $f$ is determined by substituting this solution into the
equation, and we use $\Delta t=h/30$.
}

{
The EOCs of the discretization errors and residual quantities are compared in
the first two rows of Figure \ref{fig:burgers_refinement_summary}. On the
finest meshes, the maximum error at the time-grid points and the error measured
by the $\leb{2}$-in-time dG energy seminorm have EOCs close to $p+1$ and $p$,
respectively, for all viscosities shown.
}

{
For $\veps\leq10^{-6}$, the EOCs of
$\Norm{r_1}_{\leb{1}(0,T;\leb{2}(\W))}$ and
$\Norm{\mathcal E_2}_{\leb{2}(0,T)}$ on the finest meshes are close to $p+1$
and $p$, respectively, matching those of the two discretization errors. {For
$N_e=1024$, $\mathrm{EI}_1$ and $\mathrm{EI}_2$ lie in the ranges
$1.47$--$3.15$ and $1.46$--$1.68$, respectively. At this mesh level, both
indices change only slightly for each polynomial degree as $\veps$ decreases
from $10^{-6}$ to $10^{-8}$, providing numerical evidence of robustness with
respect to viscosity in the convection-dominated regime.} By contrast, for
$\veps=10^{-1}$ and $N_e=1024$, the EOCs of
$\Norm{r_1}_{\leb{1}(0,T;\leb{2}(\W))}$ are $0.09$ for $p=1$ and $0.98$ for
$p=2$, and those of $\Norm{\mathcal E_2}_{\leb{2}(0,T)}$ are $1.00$ for $p=1$
and $2.01$ for $p=2$. {The lower EOCs of
$\Norm{r_1}_{\leb{1}(0,T;\leb{2}(\W))}$, relative to those of the maximum error
at the time-grid points, explain the growth of $\mathrm{EI}_1$ under mesh
refinement. We therefore omit the $\veps=10^{-1}$ curves from the
$\mathrm{EI}_1$ panels in Figure \ref{fig:burgers_refinement_summary}.}
}

\subsection{{Nonlinear wave system with partial diffusion}}

{
{We consider the nonlinear wave system with partial diffusion}
\begin{equation}\label{eq:nonlinear-wave}
  \begin{aligned}
    \partial_t u - \partial_x v &= f_1(t,x), \\
    \partial_t v - \partial_x W'(u) &= \veps \partial_{xx} v + f_2(t,x),
  \end{aligned}
\end{equation}
on $\W=[0,2\pi]$, subject to periodic boundary conditions. We take
$W'(u)=-u^{-\gamma}$ for $u>0$, with $\gamma=1.4$, and use the manufactured
solution
\begin{equation*}
  \begin{aligned}
    u(t,x)&=2.0+0.2\sin{2x-t}, \\
    v(t,x)&=1.0+0.3\cos{x+2t}.
  \end{aligned}
\end{equation*}
The source terms $f_1$ and $f_2$ are determined by substituting this solution
into the system, and we use $\Delta t=0.1h$.
}

{
The EOCs of the discretization errors and residual quantities are compared in
the first two rows of Figure \ref{fig:nw_refinement_summary}. On the finest
meshes, the maximum error between $(u,v)$ and $(u_h,v_h)$ at the time-grid
points and the error between $v$ and $\widehat v_h^{\mathrm t}$ measured by the
$\leb{2}$-in-time dG energy seminorm have EOCs close to $p+1$ and $p$,
respectively, for $p=1$ at all viscosities shown and for $p=2$ when
$\veps\leq10^{-6}$. For $p=2$ and $\veps\in\{10^{-1},10^{-3}\}$, both
discretization errors have EOCs close to $p$.
}

{
For $\veps\leq10^{-6}$, the EOCs of
$\Norm{\vec r_1}_{\leb{1}(0,T;\leb{2}(\W))}$ and
$\Norm{\mathcal E_2}_{\leb{2}(0,T)}$ on the finest meshes are close to $p+1$
and $p$, respectively, matching those of the two discretization errors. {For
$N_e=1024$, $\mathrm{EI}_1$ and $\mathrm{EI}_2$ lie in the ranges
$1.35$--$1.81$ and $1.58$--$1.63$, respectively. At this mesh level, both
indices change only slightly for each polynomial degree as $\veps$ decreases
from $10^{-6}$ to $10^{-8}$, providing numerical evidence of robustness with
respect to viscosity in the convection-dominated regime.} By contrast, for
$\veps=10^{-1}$ and $N_e=1024$, the EOCs of
$\Norm{\vec r_1}_{\leb{1}(0,T;\leb{2}(\W))}$ are $0.39$ for $p=1$ and $0.10$
for $p=2$, and those of $\Norm{\mathcal E_2}_{\leb{2}(0,T)}$ are $0.98$ for
$p=1$ and $1.51$ for $p=2$. {The lower EOCs of
$\Norm{\vec r_1}_{\leb{1}(0,T;\leb{2}(\W))}$, relative to those of the maximum
error at the time-grid points, explain the growth of $\mathrm{EI}_1$ under
mesh refinement. We therefore omit the $\veps=10^{-1}$ curves from the
$\mathrm{EI}_1$ panels in Figure \ref{fig:nw_refinement_summary}.}
}

\begin{figure}[p]
  \centering
  \includegraphics[width=\textwidth]{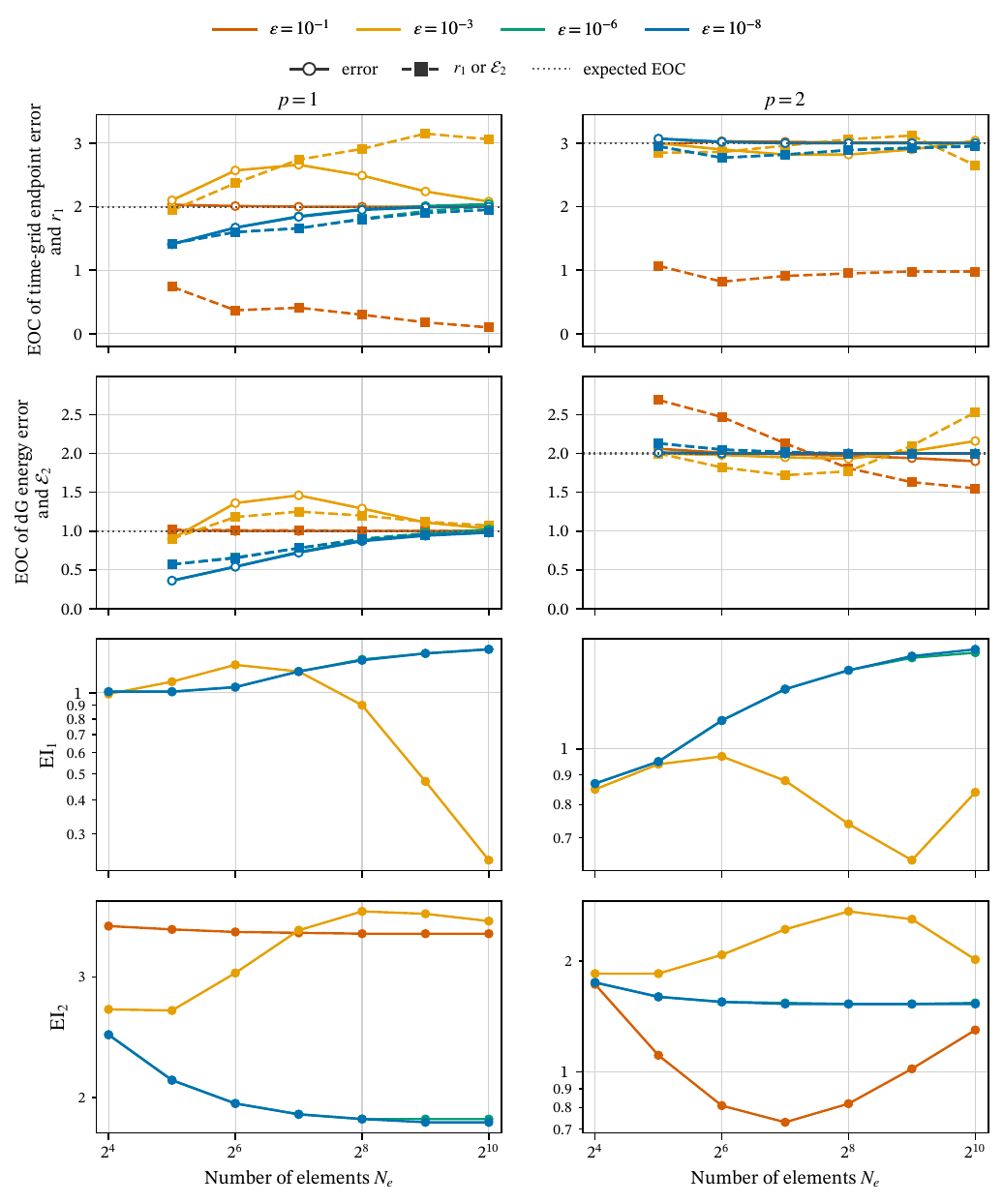}
  {
  \caption{Experimental orders of convergence and effectivity indices for the linear advection-diffusion equation \eqref{eq:linear-adv-diff} for $p=1$ (left) and $p=2$ (right). In the first two rows, solid and dashed curves represent the discretization errors and residual quantities, respectively, and dotted lines indicate the expected EOCs $p+1$ and $p$. The effectivity indices $\mathrm{EI}_1$ and $\mathrm{EI}_2$ are plotted against the number of elements $N_e$ in the last two rows.}
  \label{fig:linear_refinement_summary}
  }
\end{figure}

\begin{figure}[p]
  \centering
  \includegraphics[width=\textwidth]{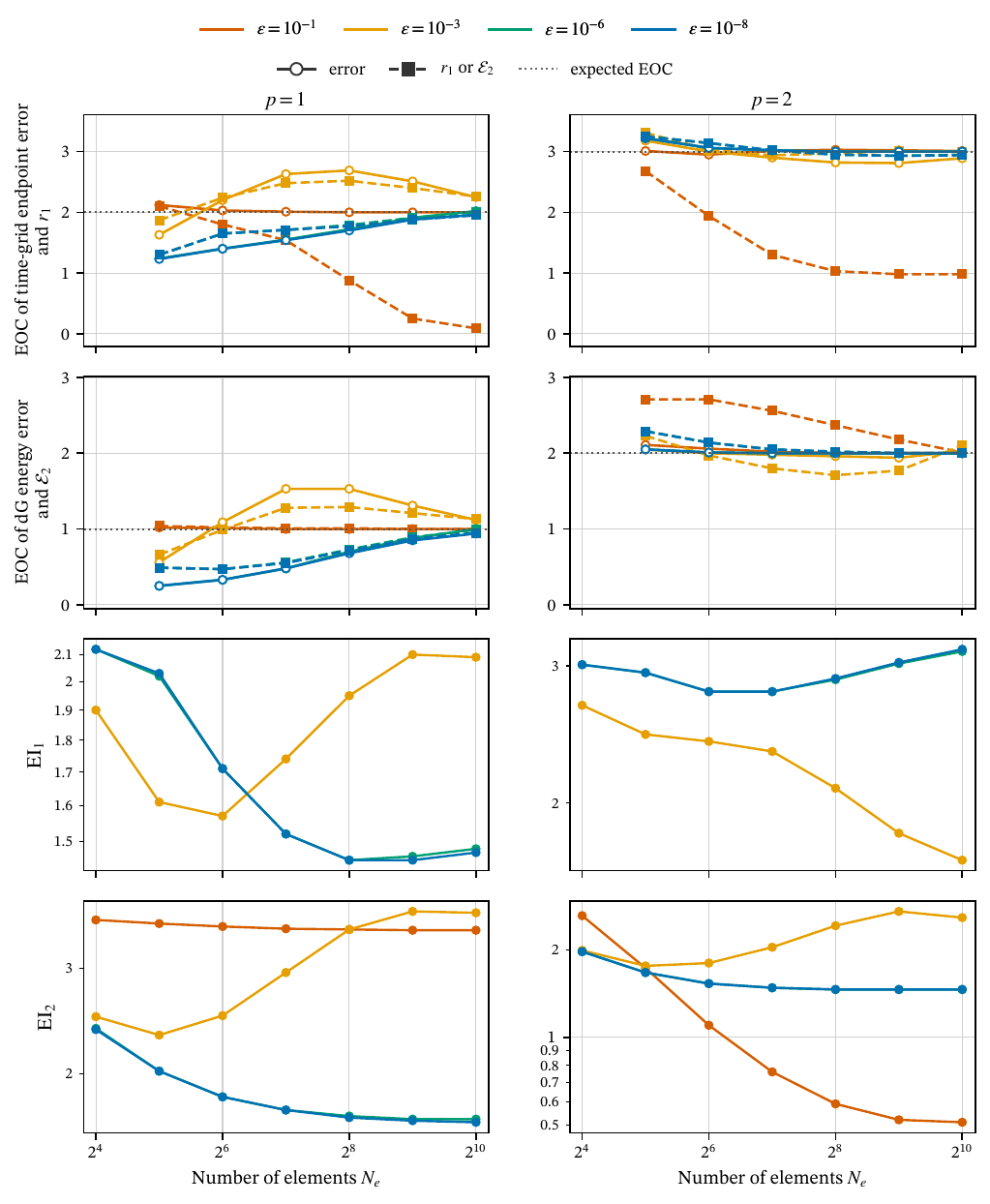}
  {
  \caption{Experimental orders of convergence and effectivity indices for the viscous Burgers equation \eqref{eq:viscous-burgers} for $p=1$ (left) and $p=2$ (right). In the first two rows, solid and dashed curves represent the discretization errors and residual quantities, respectively, and dotted lines indicate the expected EOCs $p+1$ and $p$. The effectivity indices $\mathrm{EI}_1$ and $\mathrm{EI}_2$ are plotted against the number of elements $N_e$ in the last two rows.}
  \label{fig:burgers_refinement_summary}
  }
\end{figure}

\begin{figure}[p]
  \centering
  \includegraphics[width=\textwidth]{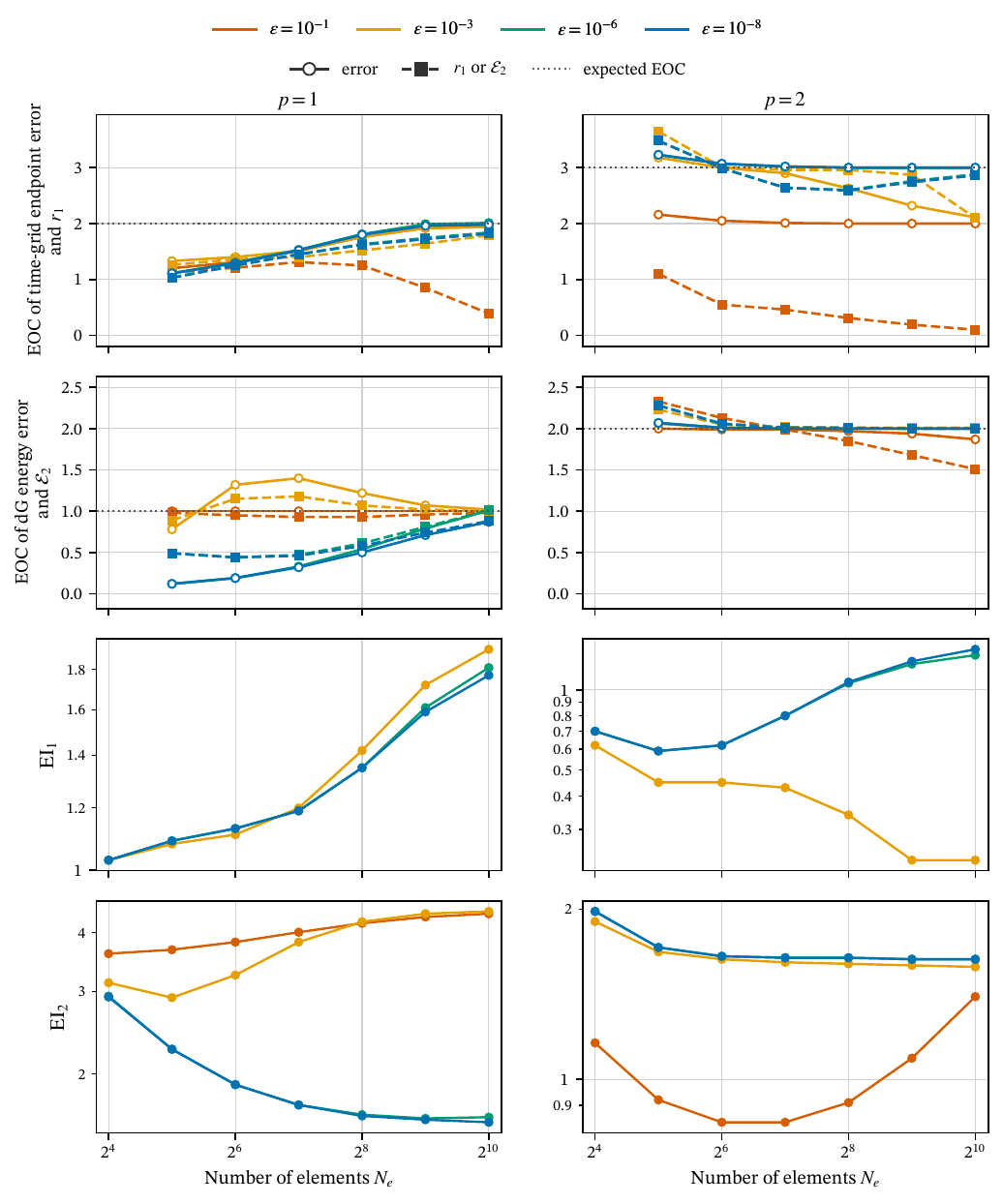}
  {
  \caption{Experimental orders of convergence and effectivity indices for the nonlinear wave system with partial diffusion \eqref{eq:nonlinear-wave} for $p=1$ (left) and $p=2$ (right). In the first two rows, solid and dashed curves represent the discretization errors and residual quantities, respectively, and dotted lines indicate the expected EOCs $p+1$ and $p$. The effectivity indices $\mathrm{EI}_1$ and $\mathrm{EI}_2$ are plotted against the number of elements $N_e$ in the last two rows. The parabolic residual estimator $\Norm{\mathcal E_2}_{\leb{2}(0,T)}$, the $\leb{2}$-in-time dG energy seminorm in the denominator of $\mathrm{EI}_2$, and $\mathrm{EI}_2$ are evaluated for the $v$-component.}
  \label{fig:nw_refinement_summary}
  }
\end{figure}
\clearpage

%%%%%%%%%%%%%%%%%%%%%%%%%%%%%%%%%%%%%%%%%%%%%%%%%%%%%%%%%%%%%%%%%%%%%%%%
%% Bibliography
\bibliographystyle{alpha}
\bibliography{adv-diff}
%%%%%%%%%%%%%%%%%%%%%%%%%%%%%%%%%%%%%%%%%%%%%%%%%%%%%%%%%%%%%%%%%%%%%%%%

\clearpage
\appendix

\section{{Estimates for boundary terms}}
\label{app:boundary-lemma}

{
\begin{lemma}
  \label{lem:boundary-est}
  Let $\W=(0,1)$, and assume that
  $\zeta \in \leb{2}(0,T; \sobh{1}_0(\W))$ and
  $\xi \in \leb{2}((0,T) \times \W)$. For every $s\in(0,T)$,
  \begin{equation}\label{boundarylemma}
    \lim_{\delta\searrow0}
    \left|
    \frac1\delta\int_0^s\int_0^\delta\zeta\xi\d x\d t
    \right|=0,
    \qquad
    \lim_{\delta\searrow0}
    \frac1{\delta^2}\int_0^s\int_0^\delta\zeta^2\d x\d t=0.
  \end{equation}
  Both limits also hold near $x=1$, with the interval $(0,\delta)$ replaced
  by $(1-\delta,1)$.
\end{lemma}
}

\begin{proof}
{
  For almost every $t\in(0,T)$, the trace of
  $\zeta(t,\cdot)\in\sobh{1}_0(\W)$ vanishes at $x=0$. Hence, for every
  $x\in(0,\delta)$, the Cauchy--Schwarz inequality gives
  \[
    |\zeta(t,x)|^2
    \leq
    x\int_0^\delta|\partial_x\zeta(t,y)|^2\d y.
  \]
  Integrating over $(0,s)\times(0,\delta)$ and taking square roots, we obtain
  \[
    \Norm{\zeta}_{\leb{2}((0,s)\times(0,\delta))}
    \leq
    \frac{\delta}{\sqrt{2}}
    \Norm{\partial_x\zeta}_{\leb{2}((0,s)\times(0,\delta))}.
  \]
  Using the bound for
  $\Norm{\zeta}_{\leb{2}((0,s)\times(0,\delta))}$ and the Cauchy--Schwarz
  inequality, we obtain
  \begin{align*}
    \left|
    \frac1\delta\int_0^s\int_0^\delta\zeta\xi\d x\d t
    \right|
    &\leq
    \frac1{\sqrt{2}}
    \Norm{\partial_x\zeta}_{\leb{2}((0,s)\times(0,\delta))}
    \Norm{\xi}_{\leb{2}((0,s)\times(0,\delta))},\\
    \frac1{\delta^2}\int_0^s\int_0^\delta\zeta^2\d x\d t
    &\leq
    \frac12
    \Norm{\partial_x\zeta}_{\leb{2}((0,s)\times(0,\delta))}^2.
  \end{align*}
  By the absolute continuity of the Lebesgue integral, both right-hand sides
  tend to zero as $\delta\searrow0$. The same estimates hold near $x=1$
  because the trace of $\zeta$ vanishes at that endpoint.
}
\end{proof}

\section{Estimates of the $H^{-1}$ norm of the parabolic component of the residual} \label{sec:comp}

{The numerical experiments in \S\ref{sec:num} do not evaluate the dual norm $\lVert \vec r_2(t)\rVert_{\sobh{-1}(\mathbb{T})}$ directly. Instead, we estimate this norm using the computable quantity $\mathcal E_2(t)$ defined in Theorem~\ref{thm:h-1_est}. We use the notation of Section 2 for jumps and averages at the mesh points $x_i$, and identify $x_0$ with $x_M$ on $\mathbb{T}$ in the periodic interface sums.}

{For the linear advection--diffusion and viscous Burgers equations, $\vec A\equiv\mathbb{I}$, and the required estimate of the parabolic residual follows from Theorem~\ref{thm:h-1_est}. For the nonlinear wave system with partial diffusion, $\vec A=\operatorname{diag}(0,1)$, and we use Theorem~\ref{thm:h-1_est} only for the scalar residual $r_{v,2}$ in the $v$-equation; see Remark B.2. Extending the estimate to state-dependent diffusion matrices would require controlling additional element and interface terms in the projection argument. Since this extension is not needed for the numerical experiments considered here, we leave it outside the scope of the present work.}

\begin{theorem}\label{thm:h-1_est}
{
For the residual decomposition $\vec r = \vec r_{1} + \veps\vec r_{2}$ from \S\ref{sec:dres}, {assume that $\vec A(\vec v)=\mathbb{I}$ for all $\vec v$} and that $\bD$ is the symmetric interior penalty operator defined in Section~\ref{sec:nsr}, with penalty matrix $\boldsymbol{\Sigma}_i=\sigma\mathbb{I}$ at every interface.
For each interface $x_i$, set
\[
\underline h_i
:=
\min\left\{
h_{i-\frac12},h_{i+\frac12}
\right\}.
\]
For almost every $t \in (0,T)$, define
\begin{align*}
  \mathcal E_2(t)
  :={}&
  \lVert \partial_x \widehat{\vec{u}}^{\mathrm{ts}}
  - \partial_x \widehat{\vec{u}}_h^{\mathrm t} \rVert_{\leb{2}(\mathbb{T})}
  +
  \left( \sum_{i=0}^{M-1} \underline h_i^{-1} |\jump{\widehat{\vec{u}}_h^{\mathrm t}}_i|^2 \right)^{1/2} \\
  &+
  \left( \sum_{i=0}^{M-1} h_i |\jump{\partial_x \widehat{\vec{u}}_h^{\mathrm t}}_i|^2 \right)^{1/2}
  .
\end{align*}
There is a constant $C>0$, independent of $h$ and $t$, such that, for almost every $t\in(0,T)$,
\begin{equation*}
  \lVert \vec r_2(t) \rVert_{\sobh{-1}(\mathbb{T})} \leq C\mathcal E_2(t),
\end{equation*}
where $C$ depends only on the constants in the trace, projection stability, and projection approximation estimates and on the penalty parameter $\sigma$.
}
\end{theorem}

\begin{remark}
{
For the nonlinear wave system in Section~\ref{sec:nw}, $\vec A=\operatorname{diag}(0,1)$ and $\vec r_2=(0,r_{v,2})^T$. We therefore apply Theorem~\ref{thm:h-1_est} to the scalar residual $r_{v,2}$ in the $v$-equation and evaluate $\mathcal E_2$ using $\widehat v^{\mathrm{ts}}$ and $\widehat v_h^{\mathrm t}$.
}
\end{remark}

\subsection{Proof of Theorem \ref{thm:h-1_est}}

{
For almost every $t\in(0,T)$, \eqref{eq:r2} and the assumption
$\vec A(\vec v)=\mathbb I$ give, for every
$\vec\phi\in\sobh{1}(\mathbb T,\reals^m)$,
\[
-\langle \vec r_2(t),\vec\phi\rangle
=
\left\langle
\partial_{xx}\widehat{\vec u}^{\mathrm{ts}}
-
\bD(\widehat{\vec u}_h^{\mathrm t}),
\vec\phi
\right\rangle .
\]
By the definition of the symmetric interior penalty operator, for every
$\vec\phi_h\in\mathbb V_h^p$,
}
\begin{equation*}
  -  \langle  \bD (\widehat{\vec{u}}_h^{\mathrm t}) , \vec \phi_h \rangle = \int_{\mathbb{T}} \partial_x \widehat{\vec{u}}_h^{\mathrm t} \cdot \partial_x \vec \phi_h{\d x}
  - \sum_{i=0}^{M-1} \Big( \jump{\widehat{\vec{u}}_h^{\mathrm t}}_i \cdot \avg{ \partial_x \vec \phi_h}_i
  + \jump{\vec \phi_h}_i \cdot \avg{ \partial_x \widehat{\vec{u}}_h^{\mathrm t}}_i - \frac{\sigma}{h_i} \jump{\widehat{\vec{u}}_h^{\mathrm t}}_i \cdot \jump{ \vec \phi_h}_i\Big),
\end{equation*}
{
where $\sigma>0$ is the penalty parameter. Let $\mathbb P$ denote the
$\leb{2}$-orthogonal projection onto $\mathbb V_h^p$.
}

{
Since $\vec\phi\in\sobh{1}(\mathbb{T},\reals^m)\subset\leb{2}(\mathbb{T},\reals^m)$, the projection $\mathbb P\vec\phi$ is well defined. Because $\bD(\widehat{\vec u}_h^{\mathrm t})\in\mathbb V_h^p$, the $\leb{2}$-orthogonality of $\mathbb P$ gives
\[
\left\langle
\bD(\widehat{\vec u}_h^{\mathrm t}),
\vec\phi
\right\rangle
=
\left\langle
\bD(\widehat{\vec u}_h^{\mathrm t}),
\mathbb P\vec\phi
\right\rangle .
\]
Taking $\vec\phi_h=\mathbb P\vec\phi$ in the SIP form and integrating the term involving $\partial_{xx}\widehat{\vec u}^{\mathrm{ts}}$ by parts on $\mathbb T$, we obtain
\begin{align*}
&
\left\langle
\partial_{xx}\widehat{\vec u}^{\mathrm{ts}}
-
\bD(\widehat{\vec u}_h^{\mathrm t}),
\vec\phi
\right\rangle \\
={}&
-\int_{\mathbb{T}}
\partial_x\widehat{\vec u}^{\mathrm{ts}}\cdot\partial_x\vec\phi{\d x}
+
\int_{\mathbb{T}}
\partial_x\widehat{\vec u}_h^{\mathrm t}\cdot
\partial_x\mathbb P\vec\phi{\d x} \\
&-
\sum_{i=0}^{M-1}
\left(
\jump{\widehat{\vec u}_h^{\mathrm t}}_i
\cdot\avg{\partial_x\mathbb P\vec\phi}_i
+
\jump{\mathbb P\vec\phi}_i
\cdot\avg{\partial_x\widehat{\vec u}_h^{\mathrm t}}_i
-
\frac{\sigma}{h_i}
\jump{\widehat{\vec u}_h^{\mathrm t}}_i
\cdot\jump{\mathbb P\vec\phi}_i
\right).
\end{align*}
Since $\vec\phi\in\sobh{1}(\mathbb{T},\reals^m)$ has no jumps at the mesh points,
\[
\jump{\vec\phi-\mathbb P\vec\phi}_i=-\jump{\mathbb P\vec\phi}_i .
\]
Applying integration by parts on each mesh interval, we obtain
\begin{align*}
&
\int_{\mathbb T}
\partial_x\widehat{\vec u}_h^{\mathrm t}
\cdot\partial_x(\vec\phi-\mathbb P\vec\phi)\d x
-
\sum_{i=0}^{M-1}
\jump{\vec\phi-\mathbb P\vec\phi}_i
\cdot\avg{\partial_x\widehat{\vec u}_h^{\mathrm t}}_i
\\
={}&
-\int_{\mathbb T}
\partial_{xx}\widehat{\vec u}_h^{\mathrm t}
\cdot(\vec\phi-\mathbb P\vec\phi)\d x
+
\sum_{i=0}^{M-1}
\jump{\partial_x\widehat{\vec u}_h^{\mathrm t}}_i
\cdot\avg{\vec\phi-\mathbb P\vec\phi}_i .
\end{align*}
The volume integral vanishes because the elementwise second derivative
$\partial_{xx}\widehat{\vec u}_h^{\mathrm t}$ belongs to $\mathbb V_h^p$ and
$\vec\phi-\mathbb P\vec\phi$ is $\leb{2}$-orthogonal to $\mathbb V_h^p$.
Using $\mathbb P\vec\phi=\vec\phi-(\vec\phi-\mathbb P\vec\phi)$, we therefore
obtain
\begin{align*}
&
\int_{\mathbb{T}}
\partial_x\widehat{\vec u}_h^{\mathrm t}\cdot
\partial_x\mathbb P\vec\phi{\d x}
-
\sum_{i=0}^{M-1}
\jump{\mathbb P\vec\phi}_i
\cdot\avg{\partial_x\widehat{\vec u}_h^{\mathrm t}}_i \\
={}&
\int_{\mathbb{T}}
\partial_x\widehat{\vec u}_h^{\mathrm t}\cdot
\partial_x\vec\phi{\d x}
-
\sum_{i=0}^{M-1}
\jump{\partial_x\widehat{\vec u}_h^{\mathrm t}}_i
\cdot\avg{\vec\phi-\mathbb P\vec\phi}_i .
\end{align*}
Substituting this identity into the formula obtained from the SIP form gives
\begin{align*}
&
\left\langle
\partial_{xx}\widehat{\vec u}^{\mathrm{ts}}
-
\bD(\widehat{\vec u}_h^{\mathrm t}),
\vec\phi
\right\rangle \\
={}&
-\int_{\mathbb{T}}
(\partial_x\widehat{\vec u}^{\mathrm{ts}}
-
\partial_x\widehat{\vec u}_h^{\mathrm t})
\cdot\partial_x\vec\phi{\d x} \\
&-
\sum_{i=0}^{M-1}
\left(
\jump{\widehat{\vec u}_h^{\mathrm t}}_i
\cdot\avg{\partial_x\mathbb P\vec\phi}_i
+
\jump{\partial_x\widehat{\vec u}_h^{\mathrm t}}_i
\cdot\avg{\vec\phi-\mathbb P\vec\phi}_i
\right)\\
&+
\sum_{i=0}^{M-1}
\frac{\sigma}{h_i}
\jump{\widehat{\vec u}_h^{\mathrm t}}_i
\cdot\jump{\mathbb P\vec\phi}_i
=:T_1+T_2+T_3 .
\end{align*}
}

{
The Cauchy--Schwarz inequality gives
\begin{equation}\label{eq:T1_est}
  |T_1| \leq \lVert \partial_x \widehat{\vec{u}}^{\mathrm{ts}} - \partial_x \widehat{\vec{u}}_h^{\mathrm t} \rVert_{\leb{2}(\mathbb{T})} \lVert \vec{\phi} \rVert_{\sobh{1}(\mathbb{T})}.
\end{equation}
}

{
We decompose $T_2$ as
\[
\begin{aligned}
T_2&=T_{2a}+T_{2b},\\
T_{2a}&:=
-\sum_{i=0}^{M-1}
\jump{\widehat{\vec u}_h^{\mathrm t}}_i
\cdot\avg{\partial_x\mathbb P\vec\phi}_i,\\
T_{2b}&:=
-\sum_{i=0}^{M-1}
\jump{\partial_x\widehat{\vec u}_h^{\mathrm t}}_i
\cdot\avg{\vec\phi-\mathbb P\vec\phi}_i.
\end{aligned}
\]
}
{
Let $C_{\text{tr}}$, $C_{\text{st}}$, and $C_{\text{app}}$ be positive constants, independent of $h$, such that, for every element $I_j$ and each endpoint $x_\ell\in\partial I_j$,
\[
\begin{aligned}
  |\vec v_h(x_\ell)|^2
  &\leq C_{\text{tr}} h_{j+\frac12}^{-1}\lVert \vec v_h\rVert_{\leb{2}(I_j)}^2,
  &&\vec v_h\in\mathbb P_p(I_j,\reals^m),\\
  \lVert \partial_x\mathbb P\vec\phi\rVert_{\leb{2}(I_j)}
  &\leq C_{\text{st}}\lVert \partial_x\vec\phi\rVert_{\leb{2}(I_j)},\\
  \lVert \vec\phi-\mathbb P\vec\phi\rVert_{\leb{2}(I_j)}
  &\leq C_{\text{app}}h_{j+\frac12}\lVert\partial_x\vec\phi\rVert_{\leb{2}(I_j)},\\
  |\vec\phi(x_\ell)-(\mathbb P\vec\phi)(x_\ell)|
  &\leq C_{\text{app}}h_{j+\frac12}^{1/2}\lVert\partial_x\vec\phi\rVert_{\leb{2}(I_j)} .
\end{aligned}
\]
The trace estimate follows from \cite[Lemma 1.46]{Di-PietroErn:2012}, the stability estimate from \cite[Lemma 1.58]{Di-PietroErn:2012} with $s=m=1$ and the triangle inequality, and the approximation estimates from \cite[Lemmas 1.58 and 1.59]{Di-PietroErn:2012} with $s=1$.
}
{
For $T_{2a}$, the triangle inequality and the Cauchy--Schwarz inequality give
\begin{equation*}
|T_{2a}| \leq \sum_{i=0}^{M-1} |\jump{\widehat{\vec{u}}_h^{\mathrm t}}_i| |\avg{\partial_x \mathbb{P}\vec{\phi}}_i|
\leq \left( \sum_{i=0}^{M-1} \underline h_i^{-1} |\jump{\widehat{\vec{u}}_h^{\mathrm t}}_i|^2 \right)^{1/2}
\left( \sum_{i=0}^{M-1} \underline h_i |\avg{\partial_x \mathbb{P}\vec{\phi}}_i|^2 \right)^{1/2}.
\end{equation*}
At each interface, Jensen's inequality gives
\begin{equation*}
\left|\avg{\partial_x\mathbb P\vec\phi}_i\right|^2
\leq \frac12\left(
\left|\partial_x\mathbb P\vec\phi(x_i^-)\right|^2
+
\left|\partial_x\mathbb P\vec\phi(x_i^+)\right|^2
\right).
\end{equation*}
Since
\[
\frac{\underline h_i}{h_{i-\frac12}}\leq1
\qquad\text{and}\qquad
\frac{\underline h_i}{h_{i+\frac12}}\leq1,
\]
the trace and stability estimates give
\begin{equation*}
\left( \sum_{i=0}^{M-1} \underline h_i |\avg{\partial_x \mathbb{P}\vec{\phi}}_i|^2 \right)^{1/2}
\leq C_{\text{st}}\sqrt{C_{\text{tr}}}\,
\lVert \vec{\phi} \rVert_{\sobh{1}(\mathbb{T})}.
\end{equation*}
Combining these estimates gives
\begin{equation}\label{eq:T2a_final}
  |T_{2a}| \leq C_{\text{st}}\sqrt{C_{\text{tr}}}
  \left( \sum_{i=0}^{M-1} \underline h_i^{-1} |\jump{\widehat{\vec{u}}_h^{\mathrm t}}_i|^2 \right)^{1/2} \lVert \vec{\phi} \rVert_{\sobh{1}(\mathbb{T})}.
\end{equation}
}

{
For $T_{2b}$, the triangle inequality and the Cauchy--Schwarz inequality give
}
\begin{align*}
|T_{2b}|
&\leq
\left( \sum_{i=0}^{M-1} h_i |\jump{\partial_x \widehat{\vec{u}}_h^{\mathrm t}}_i|^2 \right)^{1/2}\\
&\quad\times
\left( \sum_{i=0}^{M-1} h_i^{-1} |\avg{\vec{\phi} - \mathbb{P}\vec{\phi}}_i|^2 \right)^{1/2}.
\end{align*}
{
At each interface, Jensen's inequality gives
\begin{equation*}
\left|\avg{\vec\phi-\mathbb P\vec\phi}_i\right|^2
\leq\frac12\Big(
\left|(\vec\phi-\mathbb P\vec\phi)(x_i^-)\right|^2
+
\left|(\vec\phi-\mathbb P\vec\phi)(x_i^+)\right|^2
\Big).
\end{equation*}
The mesh definition gives
\[
h_i=\frac12\left(h_{i-\frac12}+h_{i+\frac12}\right),
\qquad
\frac{h_{i\pm\frac12}}{h_i}\leq2.
\]
Using $h_{i\pm\frac12}/h_i\leq2$ and the pointwise bound on the projection
error $\vec\phi-\mathbb P\vec\phi$ at the element endpoints, we obtain
\begin{equation*}
\left( \sum_{i=0}^{M-1} h_i^{-1} |\avg{\vec{\phi} - \mathbb{P}\vec{\phi}}_i|^2 \right)^{1/2} \leq \sqrt{2}C_{\text{app}}\,
\lVert \vec{\phi} \rVert_{\sobh{1}(\mathbb{T})}.
\end{equation*}
Combining these estimates gives
\begin{equation}\label{eq:T2b_final}
  |T_{2b}| \leq \sqrt{2}C_{\text{app}}
  \left( \sum_{i=0}^{M-1} h_i |\jump{\partial_x \widehat{\vec{u}}_h^{\mathrm t}}_i|^2 \right)^{1/2} \lVert \vec{\phi} \rVert_{\sobh{1}(\mathbb{T})}.
\end{equation}
}

By the Cauchy-Schwarz inequality,
\begin{align*}
|T_3|
&\leq
\sigma
\left( \sum_{i=0}^{M-1} h_i^{-1} |\jump{\widehat{\vec{u}}_h^{\mathrm t}}_i|^2 \right)^{1/2}\\
&\quad\times
\left( \sum_{i=0}^{M-1} h_i^{-1} |\jump{\mathbb{P}\vec{\phi}}_i|^2 \right)^{1/2}.
\end{align*}
{
Since
\[
\jump{\vec\phi-\mathbb P\vec\phi}_i
=
-\jump{\mathbb P\vec\phi}_i,
\]
using $h_{i\pm\frac12}/h_i\leq2$ and the pointwise bound on the projection
error $\vec\phi-\mathbb P\vec\phi$ at the element endpoints, we obtain
\begin{equation*}
\left( \sum_{i=0}^{M-1} \frac{1}{h_i} |\jump{\mathbb{P}\vec{\phi}}_i|^2 \right)^{1/2} \leq 2\sqrt{2}C_{\text{app}}\,
\lVert \vec{\phi} \rVert_{\sobh{1}(\mathbb{T})}.
\end{equation*}
Consequently,
\begin{equation}\label{eq:T3_final}
  |T_3| \leq 2\sqrt{2}\sigma C_{\text{app}}
  \left( \sum_{i=0}^{M-1} h_i^{-1} |\jump{\widehat{\vec{u}}_h^{\mathrm t}}_i|^2 \right)^{1/2} \lVert \vec{\phi} \rVert_{\sobh{1}(\mathbb{T})}.
\end{equation}
}

{
Since $\underline h_i\leq h_i$, we have
$h_i^{-1}\leq\underline h_i^{-1}$.
Combining \eqref{eq:T1_est}, \eqref{eq:T2a_final},
\eqref{eq:T2b_final}, and \eqref{eq:T3_final}, and using this inequality in
\eqref{eq:T3_final} and the definition of $\mathcal E_2(t)$, we obtain
\[
|\langle \vec r_2(t), \vec\phi\rangle|
\leq C\mathcal E_2(t)\lVert \vec\phi\rVert_{\sobh{1}(\mathbb{T})},
\]
where
\[
C
=
1
+ C_{\text{st}}\sqrt{C_{\text{tr}}}
+ \sqrt{2}C_{\text{app}}
+ 2\sqrt{2}\sigma C_{\text{app}}.
\]
Dividing by $\lVert\vec\phi\rVert_{\sobh{1}(\mathbb T)}$ and taking the
supremum over all nonzero
$\vec\phi\in\sobh{1}(\mathbb T,\reals^m)$ proves the theorem.
}

\end{document}